\documentclass[11 pt]{article}
\usepackage[T1]{fontenc}
\usepackage[utf8]{inputenc}
\usepackage{lmodern}
\usepackage{amsmath, amsthm, amssymb, amsfonts, mathrsfs}
\usepackage[normalem]{ulem} 
\usepackage{verbatim} 
\usepackage{fancyhdr}
\usepackage{pxfonts}
\usepackage{hyperref}
\usepackage{xcolor}

\newcommand\shorttitle{Uniform asymptotic formulas for theta functions}
\newcommand\authors{\small Zhi-Guo Liu and Nian Hong Zhou}

\hypersetup{
     colorlinks,
     linkcolor={red!50!black},
     citecolor={blue!50!black},
     urlcolor={blue!80!black}
}

\usepackage{longtable} 
\usepackage[center]{caption}
\usepackage{diagbox}
\usepackage{xypic}
\usepackage{mathtools}

\usepackage{tikz}
\usepackage{enumitem}
\setitemize{itemsep=1pt}
\setenumerate{itemsep=1pt}

\theoremstyle{plain}
\newtheorem{theorem}{Theorem}[section]

\newtheorem{lemma}[theorem]{Lemma}
\newtheorem{corollary}[theorem]{Corollary}
\newtheorem{proposition}[theorem]{Proposition}

\theoremstyle{remark}
\newtheorem{remark}{Remark}[section]

\makeatletter

\newcommand{\Rmnum}[1]{\expandafter\@slowromancap\romannumeral #1@}
\makeatother

\def\P{\partial}

\def\ri{\mathrm i}

\def\rb{\mathbb R}
\def\nb{\mathbb N}
\def\zb{\mathbb Z}
\def\cb{{\mathbb C}}

\def\cL{{\mathcal L}}
\def\cC{{\mathcal C}}

\def\rrw{\rightarrow}
\numberwithin{equation}{section}
\allowdisplaybreaks[2] 

\title{\Large \bf Uniform asymptotic formulas for the Fourier coefficients of the inverse of theta functions}
\author{\small Zhi-Guo Liu and Nian Hong Zhou\footnote{Corresponding author}}

\date{}

\begin{document}

\maketitle




\begin{abstract}In this paper, we use basic asymptotic analysis to establish some uniform asymptotic formulas for the
Fourier coefficients of the inverse of Jacobi theta functions. In particular, we answer and improve some
problems suggested and investigated by Bringmann, Manschot, and Dousse. As applications, we establish the
asymptotic monotonicity properties for the rank and crank of the integer partitions introduced and investigated
by Dyson, Andrews, and Garvan.
\end{abstract}

\maketitle

\section{Introduction and statement of results}
\subsection{Background}\label{sec11}
Let $\zeta=e^{2\pi\ri z}$ and $q=e^{2\pi\ri\tau}$ with $z\in\cb$ and $\Im(\tau)>0$.  Let $\vartheta(z, \tau)$ be a \emph{Jacobi theta function}
given as the following Jacobi triple product:
$$
\vartheta(z, \tau)=-\ri\zeta^{-1/2}q^{1/8}\prod_{n\ge 1}(1-q^n)(1-\zeta q^{n-1})(1-\zeta^{-1}q^n)
$$
and let $\eta(\tau)$ be the Dedekind eta function given by
$$
\eta(\tau)=\ri q^{1/6}\vartheta(\tau, 3\tau)=q^{1/24}\prod_{k\ge 1}(1-q^k).
$$
The theory of Jacobi theta functions was first introduced and studied by Jacobi \cite{jacobi_2012}, which plays an important
role in analytic and combinatorial number theory (see, e. g., Eichler and Zagier~\cite{MR781735}, Andrews \cite{MR743546} and Garvan \cite{MR929094}),
algebraic geometry (see, e. g., G\"{o}ttsche \cite{MR1032930, MR1736989}, Yoshioka, K\={o}ta \cite{MR1285785} and Hausel and Rodriguez Villegas \cite{MR3364745}) and
theoretical physics~(see, e. g., Alvarez-Gaum\'{e}, Moore and Vafa~\cite{MR853977}, Moore~\cite{MR854192} and Korpas and Manschot \cite{MR3747224}).

\medskip

Let $\nb$ denote the set of all positive integers and let $k\in\nb$. Motivated by Bringmann and Manschot \cite{MR3210725}, and Bringmann and Dousse \cite{MR3451872},
we consider the inverse of theta functions
\begin{equation}\label{eqjk}
\mathcal{J}_k \left( z , \tau \right):=\sum_{m\in\zb}\sum_{n\ge 0}j_{m,k}(n)q^n\zeta^m := \frac{\zeta^{-1/2}q^{k/24}\eta(\tau)^{3-k}}{\ri\vartheta(z,\tau)},
\end{equation}
and define numbers $a_{m,k}(n)$\footnote{In \cite{MR3451872}, $a_{m,k}(n)$ is denoted as $M_k(m,n)$.} and $b_{m,k}(n)$ by
\begin{equation}\label{eqck}
\mathcal{C}_k \left( z, \tau\right):=\sum_{m\in\zb}\sum_{n\ge 0}a_{m,k}(n)q^n\zeta^m:= (1-\zeta)\mathcal{J}_k \left( z , \tau \right),
\end{equation}
and with $\chi_m(x)=(x^m-x^{-m})/(x-x^{-1})$,
\begin{equation}\label{eqck1}
\mathcal{C}_k \left( z , \tau \right):=\sum_{m, n\ge 0}b_{m,k}(n)\chi_{2m+1}\left(\zeta^{\frac{1}{2}}\right)q^n.
\end{equation}
Clearly, for all $m\in\zb$ and $n\in\nb_0$, $a_{m,k}(n)=j_{m,k}(n)-j_{m+1,k}(n)$, and for all $m, n\in\nb_0$ $b_{m,k}(n)=a_{m,k}(n)-a_{m+1,k}(n)$.

\medskip

Follows from G\"{o}ttsche \cite{MR1032930} and Bringmann and
Manschot \cite{MR3210725}, $\cC_k(z,\tau), (k\ge 3)$
are appear in algebraic geometry, which is well-known to be generating functions of Betti numbers of moduli spaces of Hilbert schemes
on $(k-3)$-point blow-ups of the projective plane and $a_{m,k}(n)$ are the Betti numbers of the moduli spaces. Moreover, the expansion \eqref{eqck1} in
terms of $b_{m,k}(n)$ decomposes the cohomology in terms of $(2m+1)$-dimensional ${\rm SL}(2)$
or ${\rm SU}(2)_{\rm spin}$ representations. For more details,
see Bringmann and Manschot \cite{MR3210725},  Bringmann and Dousse \cite{MR3451872} and references therein.

\medskip

$\cC_k(z,\tau), (k\ge 1)$ are also generating functions of certain statistics for integer partitions, see Garvan \cite{MR920146}, Hammond and Lewis \cite{MR2102870} and Fu and Tang \cite{MR3724176}. The most well-known is the generating function $\cC_1(z,\tau)$ of \emph{crank} of integer partitions. Recall that a partition of an integer $n$ is a sequence of nonincreasing positive integers whose sum equals $n$.  Setting $z=0$ in $\mathcal{C}_k \left( z, \tau\right)$ we obtain the generating function of the number of partitions
of integer $n$ allowing $k$ colors:
\begin{equation}\label{eqck11}
\sum_{n\ge 0}p_k(n)q^{n}:=\mathcal{C}_k \left( 1; q \right)=\frac{q^{k/24}}{\eta(\tau)^{k}}.
\end{equation}
Letting $k=1$ in \eqref{eqck11} we obtain the unrestrict partitions function $p(n)=p_1(n)$.
Ramanujan's famous partition congruences \cite{MR2280868} are
$$p(5n+4)\equiv 0~(\bmod 5),$$
$$p(7n+5)\equiv 0~(\bmod 7),$$
$$ p(11n+6)\equiv 0~(\bmod {11}),$$
holding for all $n\in\nb_0$. In 1944, Dyson \cite{MR3077150} introduced the \emph{rank} statistic for integer partitions to give a combinatorial interpretation
for the above partition congruences with modulus $5$ and $7$. However, the rank fails to explain Ramanujan's congruence modulo $11$.  Therefore Dyson \cite{MR3077150} conjectured
the existence of another statistic that he called the \emph{crank} which would explain above congruence modulo $11$.
The crank was found by Andrews and Garvan \cite{MR929094, MR920146}.
\medskip


Usually, when $n\ge 2$, $M(m,n)$ is used to denote the number of partitions of $n$ with crank $m$, which equals $a_{m,1}(n)$. If we further define $M(0,1)=-1$ and $M(\pm 1,1)=1$, then
\begin{equation}\label{eqcm}
\sum_{n\ge 0}M(m,n)q^n=\frac{q^{1/24}}{\eta(\tau)}\sum_{n\ge 1}(-1)^{n-1}q^{n(n-1)/2+|m|n}(1-q^n).
\end{equation}
In 1989, Dyson \cite{MR1001259} gave the following asymptotic formula conjecture:
\begin{equation}\label{dysc0}
M\left( m,n \right)\sim \frac{\pi}{4\sqrt{6n}} {\rm sech}^2 \left( \frac{\pi m}{2\sqrt{6n}}   \right) p(n),
\end{equation}
as $n\rrw +\infty$. He then asked a problem about the precise range of $m$ in which \eqref{dysc0} holds and about the error term. This conjecture has been proved first by Bringmann and Dousse in \cite[Theorem 1.2]{MR3451872}.  The problem about the precise range has been solved by the second author of this paper in \cite{MR3924736}. Very interestingly, Dousse and Mertens in \cite{MR3337213} proved the above Dyson's Conjecture hold also for $N(m,n)$, the number of partitions of $n$ with rank $m$. Recall that for $N(m,n)$, we have
\begin{equation}\label{eqrn}
\sum_{n\ge 0}N(m,n)q^n=\frac{q^{1/24}}{\eta(\tau)}\sum_{n\ge 1}(-1)^{n-1}q^{n(3n-1)/2+|m|n}(1-q^n).
\end{equation}
See \cite{MR3210725, MR3103192,MR3279269, MR3279269, MR3565363, MR3924736, MR4000108, MR4065708} for more related investigations.

\medskip

Motivated by the theory of integer partitions, algebraic geometry and theoretical physics, Bringmann and Manschot \cite{MR3210725} and
Bringmann and Dousse \cite{MR3451872} investigated the asymptotics of $a_{m,k}(n)$ and $b_{m,k}(n)$. In \cite{MR3210725} they proved the asymptotics for
the cases of fixed $m$, and in \cite{MR3451872} proved the uniform asymptotic formula for $a_{m,k}(n)$ in which holds for all $|m|\le \sqrt{{n}/{(6k\pi^2)}}\log n$
as $n$ tends to infinity.  For more related investigations, see for examples, \cite{MR855209, MR3361299, MR3906343, MR4065708}.

\medskip

In this paper, we investigate the more precise uniform asymptotic behavior of $j_{m,k}(n)$, $a_{m,k}(n)$ and $b_{m,k}(n)$,
which is partly suggested by Bringmann and Manschot \cite[Section 1.2]{MR3210725}. Our main tool is the classical asymptotic analysis which
is different above literature \cite{MR3210725, MR3451872, MR3361299, MR3906343}. We note that all of the above literature used the classical circle method.

\subsection{Main results}

The first result of this paper is the following uniform asymptotic formulas for $j_{m,k}(n)$, $a_{m,k}(n)$ and $b_{m,k}(n)$.
Let ${\bf 1}_{event}$ be the indicator function. We prove the following theorem.
\begin{theorem}\label{znh1} Let $m\in\zb$ and $n\in\nb$ such that $m= o(n^{3/4})$. We have
$$\frac{j_{m,k}\left(n-m{\bf 1}_{m>0}\right)}{p_k(n)}= \frac{1}{2}\left(1-\tanh\left(\frac{2|m|-1}{4}\sqrt{\frac{k\pi^2}{6n}}\right)\right)
\left(1+O\left(\frac{n+m^2}{n^{3/2}}\right)\right),$$
$$\frac{a_{m,k}(n)}{p_k(n)}= \frac{1}{4}\sqrt{\frac{k\pi^2}{6n}}{\rm sech}^2\left(\frac{m}{2}\sqrt{\frac{k\pi^2}{6n}}\right)
\left(1+O\left(\frac{n+m^2}{n^{3/2}}\right)\right),$$
and
\begin{align*}
\frac{b_{m,k}(n)}{p_k(n)}=&\frac{k\pi^2}{24n}
{\rm sech}^2 \left(\frac{2m+1}{4}\sqrt{\frac{k\pi^2}{6n}}\right)\tanh\left(\frac{2m+1}{4}\sqrt{\frac{k\pi^2}{6n}}\right)
\left(1+O\left(\frac{n+m^2}{n^{3/2}}\right)\right).
\end{align*}
\end{theorem}
The uniform asymptotic formula for $a_{m,k}(n)$ holds for more wide-ranging $m$ and smaller error term than \cite[Theorem 1.4]{MR3451872} of Bringmann and Dousse.
The uniform asymptotic formulas for $j_{m,k}\left(n-m{\bf 1}_{m>0}\right)$ and $b_{m,k}(n)$ are new. The formula for $b_{m,k}(n)$ also shows that there exists a constant $A_k$
such that $b_{m,k}(n)$ increases with $m\le ({1}/{2\pi})\sqrt{{6n}/{k}}\log(2+\sqrt{3})-A_k$ for all sufficiently large $n$.

We also prove the following more widely unform asymptotics.
\begin{theorem}\label{znh2}With $\delta_k(n)=\sqrt{k\pi^2/6n}$ and uniformly for all $m\in\zb$, as $n\rrw+\infty$
$$\left(1+e^{-|m|\delta_k(n)}\right)j_{m,k}\left(n+|m|{\bf 1}_{m<0}\right)\sim p_{k}(n),$$
$$\left(1+e^{-|m|\delta_k(n)}\right)^2a_{m,k}(n+|m|)\sim \delta_k(n)p_k(n),$$
and
$$\left(1+e^{-|m|\delta_k(n)}\right)^2b_{m,k}(n+|m|)\sim \delta_k(n)^2\tanh\left(4^{-1}(2m+1)\delta_k(n)\right)p_k(n).$$
\end{theorem}

From the Jacobi triple product, we know that $\mathcal{J}_k \left( z , \tau \right)$ is a meromorphic function of $\zeta\in\cb$, all poles are simple and
in the form of $q^{\ell}, \ell\in\zb$. Using the Mittag-Leffler theorem, $\mathcal{J}_k \left( z , \tau \right)$ can be represented as a so-called Lerch sum,
\begin{equation}\label{eqjkl}
\mathcal{J}_k \left( z , \tau \right)=\frac{q^{k/24}}{\eta(\tau)^k} \sum_{n\in\zb} \frac{(-1)^nq^{\frac{n(n+1)}2}}{1-\zeta q^n}, \quad (\zeta\not\in q^{\zb}),
\end{equation}
see Ramanujan's lost notebook \cite[Entry~3.2.1]{MR2952081} or Garvan \cite[Equation~(7.15)]{MR920146} for details. From \eqref{eqjk} and \eqref{eqjkl},
the Fourier coefficients $j_{m,k}(n)$ have the following generating function:
\begin{equation}\label{eqjkf}
\sum_{n\ge 0}j_{m,k}(n)q^n=\frac{q^{k/24}}{\eta(\tau)^k}\sum_{n\ge 1}(-1)^{n-1}q^{\frac{1}{2}n^2+(|m|-\frac{1}{2})n-|m|{\bf 1}_{m>0}},
\end{equation}
for all $m\in\zb$. Therefore, from the expansions \eqref{eqjkf} for $j_{m,k}(n)$ and \eqref{eqck11} for $p_k(n)$,
for all $m\in\zb$ and $n\in\nb$ we have
\begin{equation}\label{eqjef}
j_{m,k}\left(n-m{\bf 1}_{m>0}\right)=\sum_{\substack{\ell\ge 1\\ \frac{1}{2}\ell^2+(|m|-\frac{1}{2})\ell\le n}}(-1)^{\ell-1}p_k\left(n-\left(\frac{1}{2}\ell^2
+\left(|m|-\frac{1}{2}\right)\ell\right)\right).
\end{equation}

On the other hand, motivated by Dyson \cite{MR1001259, MR0238711}, Garvan \cite{MR1291125} and Berkovich and Garvan \cite{MR1932070}, let $I_k(m,n), (k\ge 1)$ be given by
\begin{equation}\label{1grk}
{\rm FG}_{k,m}(q):=\sum_{n\ge 0}I_{k}(m,n)q^n:=\frac{q^{1/24}}{\eta(\tau)}\sum_{n\ge 1}(-1)^{n-1}q^{n((2k-1)n-1)/2+mn},
\end{equation}
for all $m, n\in\nb_0$, and $N_k(m,n)$ be defined as
$$
\sum_{n\ge 0}N_{k}(m,n)q^n:={\rm FG}_{k,|m|}(q)-{\rm FG}_{k,|m|+1}(q),
$$
for all $m\in\zb$ and $n\in\nb_0$. From \eqref{eqcm}, \eqref{eqrn} and above, it is clear that
\begin{equation}\label{2grk}
N_k(m,n)=I_{k}(|m|,n)-I_{k}(|m|+1,n),
\end{equation}
$N_1(m,n)=M(m,n)$ and $N_2(m,n)=N(m,n)$. The integers $I_{k}(m,n)$ has a generating function similar to $j_{m,k}(n)$, that is \eqref{eqjkf}. Moreover, Dyson \cite{MR1001259, MR0238711} and Berkovich and Garvan \cite{MR1932070} showed that $I_k(m,n)$ are count certain statistics of integer partitions.
For each integer $k\ge 3$, Garvan \cite[Theorem (1.12)]{MR1291125} proved that
$N_k(m, n)$ is the number of partitions of $n$ into at least $(k-1)$ successive Durfee squares with \emph{$k$-rank} equals $m$.

For the above reasons, we can also use the same argument to Theorem \ref{znh1} and Theorem \ref{znh2} to give the uniform asymptotics for the rank and crank statistics for integer partitions. We have the following uniform asymptotics for $I_k(m,n)$ and $N_k(m,n)$.
\begin{theorem}\label{znh3} Let $k,n\in\nb$ and  $m\in\zb$ such that $m= o(n^{3/4})$. We have
$$\frac{I_k(|m|,n)}{p(n)}= \frac{1}{2}\left(1-{\rm tanh}\left(\frac{\pi(2|m|-1)}{4\sqrt{6n}}\right)\right)\left(1+
O\left(\frac{n+m^2}{n^{3/2}}\right)\right),$$
\begin{align*}
\frac{N_k(m,n)}{p(n)}=\frac{\pi}{4\sqrt{6n}}{\rm sech}^2\left(\frac{\pi m}{2\sqrt{6n}}\right)\left(1+O\left(\frac{n+m^2}{n^{3/2}}\right)\right)
\end{align*}
and
\begin{align*}
\frac{N_k(m,n)-N_k(m+1,n)}{p(n)}
= &\frac{\pi^2}{24n}{\rm sech}^2\left(\frac{\pi(2m+1)}{4\sqrt{6n}}\right)\\
&\times\tanh \left(\frac{\pi(2m+1)}{4\sqrt{6n}}\right)
\left(1+O\left(\frac{n+m^2}{n^{3/2}}\right)\right).
\end{align*}
\end{theorem}
\begin{remark}
The asymptotic formula for $N_2(m,n)$ improves the result of Dousse and Mertens \cite{MR3337213}.
\end{remark}
We also have the following unform asymptotics.
\begin{theorem}\label{znh4}With $\delta(n)=\sqrt{\pi^2/6n}$ and uniformly for all $m\in\zb$, as $n\rrw+\infty$
$$\left(1+e^{-|m|\delta(n)}\right)I_{k}\left(|m|, n\right)\sim p(n),$$
$$\left(1+e^{-|m|\delta(n)}\right)^2N_k(m,n+|m|)\sim \delta(n)p(n),$$
and
$$\frac{N_k(m,n+|m|)-N_k(m+1,n+|m|)}{\delta(n)^2p(n)}\sim \frac{\tanh\left(4^{-1}(2m+1)\delta(n)\right)}{\left(1+e^{-|m|\delta(n)}\right)^2}.$$
\end{theorem}
The third asymptotic formula of Theorem \ref{znh4} gives the asymptotic monotonicity properties for $N_k(m,n)$.
We note that the monotonicity properties of $N_2(m,n)$ and $N_1(m,n)$ has been investigated by Chan and Mao \cite{MR3190432} and Males \cite{JMal}, and Ji and Zang \cite{JZ}, respectively.

Furthermore, from \eqref{1grk} and \eqref{2grk} we have for all $|m|>n/2$,
$N_k(m,n)=p(n-(k-1+|m|))-p(n-(k+|m|))$.
This yields for all $m>n/2$ and with $\ell=n-k-m$ we have
$$N_k(m,n)-N_k(m+1,n)=p(\ell+1)-2p(\ell)+p(\ell-1).$$
Using the fact that $p(\ell+1)-2p(\ell)+p(\ell-1)\ge 0$ for all $\ell>0$, we have
$N_k(m,n)-N_k(m+1,n)\ge 0$ for all $n/2<m< n-k$. Therefore, using Theorem \ref{znh4} and the fact that $N_k(m,n)=N_k(|m|,n)$ we have
following unimodal properties:
\begin{corollary}\label{corocm1}
For each $k\in\nb$, $\{N_k(m,n)\}_{m=k+1-n}^{n-k-1}$ is a unimodal sequence for all sufficiently large positive integers $n$.
\end{corollary}
We are also interested in the difference between $I_{k}(m,n)$ and $I_{k+1}(m,n)$, and $N_{k}(m,n)$ and $N_{k+1}(m,n)$. We derive the following uniform asymptotic formulas.
\begin{theorem}\label{znh5} Let $k,n\in\nb$ and $m\in\zb$ such that $m= o(n^{3/4})$. We have
\begin{align*}\frac{I_{k}(|m|,n)-I_{k+1}(|m|,n)}{p(n)}
=&\frac{\pi}{4\sqrt{6n}}{\rm sech}^2\left(\frac{\pi (2|m|-1)}{4\sqrt{6n}}\right)\\
&\times\left(\tanh\left(\frac{\pi (2|m|-1)}{4\sqrt{6n}}\right)
+O\left(\frac{n^{2}+m^4}{n^{5/2}}\right)\right)
\end{align*}
and
$$\frac{N_{k+1}(m,n)-N_{k}(m,n)}{\pi^2(48n)^{-1}p(n)}
={\rm sech}^2\left(\frac{\pi m}{2\sqrt{6n}}\right)\left(1-3\tanh^2\left(\frac{\pi m}{2\sqrt{6n}}\right)
+O\left(\frac{n^{2}+m^4}{n^{5/2}}\right)\right).$$
In particular, if $m=m_k(n)\in\zb$ with $|m_k(n)|=O(\sqrt{n})$ such
that $|N_k(m,n)-N_{k+1}(m,n)|$ takes the minimum value, then as $n\rrw +\infty$
$$|m_k(n)|=\frac{\log(2+\sqrt{3})}{\pi}\sqrt{6n}+O(1).$$

\end{theorem}

\begin{remark}
From Theorem \ref{mth} below, it is possible to prove an asymptotic expansion of the following form for the above $m_k(n)$:
$$|m_k(n)|\sim \frac{\sqrt{6n}}{\pi}\log(2+\sqrt{3})+\sum_{d\ge 0}\frac{c_k(d)}{n^{d/2}},$$
as $n\rrw +\infty$, where $c_k(d)$ are computable constants depending only on $k$ and $d$.
\end{remark}

We formulate the idea of the proof of our main results of this paper as the following. Let us consider a class of functions $f$ which has a similar asymptotic expansion to the partition functions $p_k(n)$, that is
\begin{equation}\label{aspk}
p_k(n)\sim \frac{e^{2\pi\sqrt{{kn}/{6}}}}{n^{\alpha_{p_k}}}\sum_{\ell\ge 0}\frac{\gamma_{\ell}(p_k)}{n^{\ell/2}},
\end{equation}
as $n\rrw +\infty$, for some constants $\alpha_{p_k}, \gamma_{\ell}(p_k)\in\rb$ and $\gamma_{0}(p_k)\in\rb_+$.
We note that the asymptotic expansion \eqref{aspk} was essentially proved by Rademacher and Zuckerman \cite{MR1503417}. More precisely, let $f: \rb\rrw\rb$ be a real
function with asymptotic expansion of the form
\begin{equation}\label{fdef}
f(X)\sim \frac{e^{\beta_f\sqrt{X}}}{X^{\alpha_f}}\sum_{n\ge 0}\frac{\gamma_n(f)}{X^{n/2}},
\end{equation}
as $X\rrw +\infty$, with $\beta_f\in\rb_+$, $\alpha_f, \gamma_n(f)\in\rb$ for all integers $n\ge 0$ and $\gamma_0(f)\in\rb_+$.
Also, suppose that if $x<0$ then $f(x)=0$, and for any given $A>0$, $f(x)=O(1)$ for all $|x|\le A$. Let us define $S_{f}(a,b; X)$ by the following alternating sum
\begin{equation}\label{eq1}
S_{f}(a,b; X):=\sum_{\substack{n\ge 1\\ an^2+bn\le X}}(-1)^{n-1}f\left(X-(an^2+bn)\right),
\end{equation}
with $a\in\rb_+$ and $b\in\rb$. Then, from \eqref{eqjkf}--\eqref{1grk}, it is clear that
\begin{equation}\label{eq00}
j_{m,k}(n-m{\bf 1}_{m>0})=S_{p_k}\left({1}/{2},|m|-{1}/{2}; n\right)
\end{equation}
and
\begin{equation}\label{eq00'}
I_k(|m|,n)=S_{p}\left(k-{1}/{2},|m|-{1}/{2}; n\right).
\end{equation}
Therefore Theorems \ref{znh1}--\ref{znh5} follow from the asymptotics for $S_{f}\left(a,b; X\right)$.

\medskip

Our main results of this paper are stated in the following.
\begin{theorem}\label{mth}
\label{th20}Let $p\in\nb, a\in\rb_+$ and $\mu\in\rb$ be given. For $b, X\ge 0$ such that $b= o(X^{3/4})$ we have
\begin{equation*}
\frac{S_{f}(a,b+\mu;X)}{f(X)}=\left(\sum_{0\le d<3p }\frac{\cL_{f,d}(\mu ,a,b,\P_{\alpha})}{X^{d/2}}
+O\left(\frac{X^{-p}+\alpha^{2p}}{X^{p/2}}\right)\right)\bigg|_{\alpha=\frac{b\beta_f}{2\sqrt{X}}}\left\{\frac{1}{1+e^{\alpha}}\right\},
\end{equation*}
as $X\rrw \infty$. Here $\P_{\alpha}=\frac{\,d}{\,d\alpha}$,
$\cL_{f,d}(\mu,a,b,\P_{\alpha})$ is a differential operator defined as
$$\cL_{f,d}(\mu,a,b,\P_{\alpha})
=\sum_{\substack{r,\ell,s\ge 0\\ 3r+2\ell+2s\le 2d}}C_{r,\ell,s}(d;f)a^sb^r\mu^{\ell}\P_{\alpha}^{r+\ell+2s},$$
and $C_{r,\ell,s}(d;f)$ are constants given by \eqref{eqccc} depending only on $r$, $\ell$,
$s$, $d$, and $f$.
\end{theorem}
\begin{remark}
A parameter $\mu$ was introduced in the above theorem. So that we can conveniently compute the difference of $S_f(a,b; X)$ with respect to $b$.
This is required to find the asymptotics for $a_{m,k}(n)=j_{m,k}(n)-j_{m+1,k}(n)$, $b_{m,k}(n)=a_{m,k}(n)-a_{m+1,k}(n)$, $N_k(m,n)=I_{k}(|m|,n)-I_{k}(|m|+1,n)$ and others.
\end{remark}
For  relatively large $b$ we prove the following theorem.
\begin{theorem}\label{thm1}Let $L(x)$ be a real function satisfying $\lim\limits_{x\rrw+\infty}L(x)=+\infty$, and
let $a\in\rb_+$ and $\mu\in\rb$ be given. If $b>X/3$ and $X\rrw+\infty$ then
$$S_f(a,b; X)=f(X-a-b)+O(|f(X-4a-2b)|).$$
If $X, b\in\rb_+$ such that $L(X)X^{1/2}\log X\le b\le X-L(X)$, then we have an
asymptotic expansion of form
$$
\frac{S_f(a,b+\mu; X)}{f(X-b)}\sim \sum_{d\ge 0}\frac{1}{(X-b)^{d/2}}\sum_{\substack{\ell,s\ge 0\\ \ell+s\le d}}C_{\ell,s}(d;f)\mu^{\ell}a^{s},
$$
where $C_{\ell,s}(d;f)$ are constants given by \eqref{eqccc1} depending only on $\ell$,
$s$, $d$ and $f$.
\end{theorem}

We shall give some asymptotic formulas for Theorem \ref{mth} and Theorem \ref{thm1} which are convenient to use. The leading asymptotic behavior involving $S_f(a,b; X)$ follows from those formulas. To do this, we need to introduce the forward
difference operator $\Delta$ defined as
$$\Delta_x^{0}F(n)=F(x),~\Delta_x F(x)=F(x+1)-F(x), ~\Delta_x^{J+1}F(x)=\Delta_x(\Delta_x^{J}F(x)),$$
for each $J\in\nb$ and any given function $F(x)$.  We prove the following theorem.

\begin{theorem}\label{propm}Let $J\in\nb_0$, $a, a_1,a_2\in\rb_+$ and $\mu\in\rb\setminus\{0\}$ be given.
For $b, X\ge 0$ such that $b= o(X^{3/4})$ we have
\begin{align*}
&\frac{\Delta_u^J\big|_{u=0}S_f(a,b+u\mu;X)}{X^{-J/2}(\mu \beta_f/2)^Jf(X)}\\
&\qquad\quad\qquad=\bigg(\P_{\alpha}^J-\frac{{M}_{f,J}(a,\P_{\alpha})}{2\beta_f\sqrt{X}}+O\left(\frac{1+\alpha^4}{X}\right)\bigg)
\Big|_{\alpha=\frac{(2b+\mu J)\beta_f}{4\sqrt{X}}}\left\{\frac{1}{1+e^{\alpha}}\right\},
\end{align*}
as $X\rrw +\infty$, where $\P_{\alpha}:=\frac{\,d}{\,d\alpha}$ and
\begin{align*}
{M}_{f,J}(a,\P_{\alpha})=\left(4\alpha_f-1+J\right)J\P_{\alpha}^{J}
+2(J+2\alpha_f)\alpha\P_{\alpha}^{J+1}+(a\beta_f^2+\alpha^2)\P_{\alpha}^{J+2}
\end{align*}
is a differential operator. In particular,
\begin{align*}
&\frac{\Delta_{r}\big|_{r=0}\Delta_{u}^J\big|_{u=0}S_{f}(a_{r},b+\mu u;X)}{X^{-J/2}(\mu \beta_f/2)^Jf(X)}\\
&\qquad\qquad=\left(\frac{(a_0-a_1)\beta_f}{2\sqrt{X}}\P_{\alpha}^{J+2}+ O\left(\frac{1+\alpha^4}{X}\right)\right)
\Big|_{\alpha=\frac{(2b+\mu J)\beta_f}{4\sqrt{X}}}\left\{\frac{1}{1+e^{\alpha}}\right\},
\end{align*}
as $X\rrw +\infty$.
\end{theorem}
From Theorem \ref{thm1} and Theorem \ref{propm} we have
\begin{theorem}\label{propm'}Let $J\in\{0,1,2\}$, $a,c\in\rb_+$ and $\mu\in\rb\setminus\{0\}$ be given.
For $b, X\ge 0$ such that $|2b+\mu J|\ge c$ we have
\begin{align*}
\frac{\Delta_u^J\big|_{u=0}S_f(a,b+u\mu;X+b)}{X^{-J/2}(\mu \beta_f/2)^Jf(X)}
\sim e^{\alpha}\P_{\alpha}^J
\Big|_{\alpha=\frac{(2b+\mu J)\beta_f}{4\sqrt{X}}}\left\{\frac{1}{1+e^{\alpha}}\right\},
\end{align*}
as $X\rrw +\infty$.
\end{theorem}

\subsection{Organization of the paper}
This paper is organized as follows. In Section \ref{sec2}, we prove some results on the asymptotics of $f(X)$ defined by \eqref{fdef} and
the \emph{false theta functions} defined by \eqref{PTF}.
In Section \ref{sec3}, we prove Theorems \ref{mth} and \ref{thm1}. In Section \ref{sec4}, we prove Theorems \ref{propm} and \ref{propm'}.
In Section \ref{sec5}, we prove Theorems \ref{znh1}--\ref{znh5}, and illustrate results for the cases
of $b_{m,1}(n)$ and $N_{2}(m,n)-N_{2}(m+1,n)$ numerically.

\medskip

\paragraph{\bf Notations.}The symbols $\nb, \nb_0$, $\rb$ and $\rb_+$ denote the set
of positive integer, nonnegative integer, real and positive real numbers, respectively. $\P_{\alpha}=\frac{\,d}{\,d\alpha}$ is
the usual derivative operator. We use $Y=O(X)$ or $Y\ll X$ to
denote $Y\le C X$, and $Y\gg X$ to denote $Y\ge C X$, for some absolute real number $C\in\rb_+$.


\section{Primarily}\label{sec2}
In this section, we employ the definitions and elementary properties from \cite[Section 2.1]{NIST:DLMF} on asymptotic expansions.
We investigate some asymptotic properties of $f(X)$ defined by \eqref{fdef} and the asymptotics of the following false theta function:
\begin{equation}\label{PTF}
T_{a, b}(z)=\sum_{n\ge 1}(-1)^{n-1}e^{-(an^2+bn)z},
\end{equation}
where $a\in\rb_+$, $b\in\rb$ and $z\in\cb$ with $\Re(z)>0$.

We give some explanations for the main results of this section. In view of \eqref{eq1} and our aim of this paper, the intentions of Proposition \ref{lem22} and
Corollary \ref{cor21} in the following are necessary. Now, if we using Proposition \ref{lem22} in \eqref{eq1} then we find that we need the uniform asymptotics
of $T_{a, b}(z)$, see Proposition \ref{pro31} of Section \ref{sec3} below. More precisely, we need uniform asymptotics in which uniformly for all $b$ such
that $0\le b=o(z^{-3/2})$, as real $z\rrw 0^+$. Hence such uniform asymptotics stated in Theorem \ref{thmft} below plays a crucial role in the proof of our main
theorem Theorem \ref{mth}.

The false theta function \eqref{PTF} has recently appeared in several areas of the theory of $q$-series, integer partitions
and quantum topology. In all of these aspects, it is important to understand the asymptotic properties of \eqref{PTF}.

At present, the tools to obtain the asymptotics for \eqref{PTF} are the Euler--Maclaurin summation formula and Mellin transform. We reference Zagier \cite{MR2257528} and
Bringmann et al. recent work \cite{MR3210725, MR3597015, BJM20} on the Euler--Maclaurin summation formula, and Berndt and Kim \cite{MR3103192} and Mao \cite{MR2823024} on the Mellin transform.

However, the above literature just deals with the asymptotic expansion of false theta function \eqref{PTF} with $a, b$ fixed as $z\rrw 0$. In fact, from \cite[Corollary 5.1]{BJM20} it is easy to
find that
\begin{equation}\label{eqft}
2e^{-b^2z/4a}(1-T_{a,b}(z))\sim \sum_{n\ge 0}\frac{E_{n}(b/2a)H_{n}(0)}{n!}(az)^{n/2},
\end{equation}
for given real numbers $a>0, b\in\rb$, as $z\rrw 0^+$. Here $E_n(\cdot)$ is the $n$-th \emph{Euler polynomial} which has degree $n$ and
$$H_n(0)=\P_{t}^n\big|_{t=0}e^{-t^2}=\frac{(-1)^{n/2}n!}{(n/2)!}{\bf 1}_{\rm n~ even}.$$
is the $n$-th \emph{Hermite number}. In particular, for any given integer $n\ge 0$ and real number $z, a>0$,
$$\frac{E_{2n}(b/2a)H_{2n}(0)}{(2n)!}(az)^{n}\sim \left(\frac{b}{2a}\right)^{2n}\frac{(-az)^n}{n!}=\frac{(-1)^n}{n!(4a)^n}(b^2z)^n,$$
as $b\rrw \infty$. The best, we can expect that the asymptotic expansion \eqref{eqft} holds uniformly for all real number $b,z>0$ such that $b^2z\rrw 0$, that is $b=o(z^{-1/2})$.
However, such a result is still not enough for our purpose because we need asymptotics that holds for all $b=o(z^{-3/2})$.

Luckily, since the series expansion \eqref{PTF} is alternating, we can employ the classical Taylor theorem and Euler transform (Lemma \ref{lem41}) to find the
asymptotics (Theorem \ref{thmft}) which uniformly holds all $b\ge 0$.  See Subsection \ref{sec22} for details.

\subsection{Shift of the asymptotic expansion of \texorpdfstring{$f(X)$}{Lg}}

We begin with the following asymptotic result for a shift of an asymptotic expansion, which will be used to deduce the approximation for
$f(X+r)$ with $r=o(X^{3/4})$.
\begin{proposition}\label{lem22}If $r=o(X^{3/4})$ then $f(X+r)$ has an asymptotic expansion of the form
\begin{align*}
\frac{f(X+r)}{f(X)}=\left(\sum_{0\le j<p}X^{-3j/4}\Lambda_{j}(f,X)\P_y^j+O\left(\left|\frac{r}{X^{3/4}}\right|^p\right)\right)
\bigg|_{y=\frac{\beta_f}{2\sqrt{X}}}e^{y r},
\end{align*}
as $X\rrw +\infty$, for each $p\ge 1$. Here $\Lambda_{j}(k,X)$ has an asymptotic expansion of the form
$$\Lambda_{j}(f,X)\sim \sum_{n\ge 0}\lambda_{n,j}(f)X^{-n/4}$$
for some constants $\lambda_{n,j}(f)\in\cb$ depending only on $n$, $j$ and $f$ are defined as \eqref{ldef}. In particular,
$\lambda_{0,0}(f)=1,\lambda_{n,0}(f)=0~\mbox{for all}~ n\ge 1,$ $\lambda_{1,1}(f)=
-\alpha_f$, $\lambda_{0,2}(f)=-{\beta_f}/{8}$ and $\lambda_{n,j}(f)=0$ for all nonnegative integers $n$ and $j$ such that $n\not\equiv j\bmod 2$.
\end{proposition}
\begin{proof}
First of all, since $r=o(X^{3/4})$, we have
$$\sqrt{X}\left(\sqrt{1+\frac{r}{X}}-1-\frac{r}{2X}\right)=\sqrt{X}\sum_{i\ge 2}\binom{1/2}{i}\left(\frac{r}{X}\right)^i
=O\left(\frac{r^2}{X^{3/2}}\right)=o(1),$$
by the generalized binomial theorem, and hence we have
\begin{align}\label{eq22}
e^{\beta_f\sqrt{X}\left(\sqrt{1+\frac{r}{X}}-1-\frac{r}{2X}\right)}&=\sum_{\ell\ge 0}\frac{(\beta_fX^{1/2})^{\ell}}{\ell !}
\left(\sum_{i\ge 2}\binom{1/2}{i}\left(\frac{r}{X}\right)^i\right)^{\ell}\nonumber\\
&=:\sum_{\ell\ge 0}r^{2\ell}X^{-3\ell/2}\sum_{k\ge 0}d_{k,\ell}(f)\left(\frac{r}{X}\right)^k,
\end{align}
by using the Taylor expansion of $e^x$. Also, since $|r/X|=o(1)$, we have
$$(X+r)^{-n/2-\alpha_f}=X^{-n/2-\alpha_f}\sum_{d\ge 0}\binom{-n/2-\alpha_f}{d}\left(\frac{r}{X}\right)^d,$$
by the generalized binomial theorem, and if we define $c_{d,i}(f)$ for each $d, i\in\nb_0$ that
\begin{equation}\label{gke}
\sum_{i\ge 0}c_{d,i}(f)X^{-i/2}:=\left(\sum_{n\ge 0}\frac{\gamma_n(f)}{X^{n/2}}\right)^{-1}
\left(\sum_{n\ge 0}\binom{-n/2-\alpha_f}{d}\frac{\gamma_n(f)}{X^{n/2}}\right),
\end{equation}
formally, then
\begin{equation}\label{eq23}
\left(\sum_{n\ge 0}\frac{\gamma_n(f)}{X^{n/2+\alpha_f}}\right)^{-1}
\sum_{n\ge 0}\frac{\gamma_n(f)}{(X+r)^{n/2+\alpha_f}}\sim \sum_{d\ge 0}\left(\frac{r}{X}\right)^{d}\sum_{i\ge 0}c_{d,i}(f)X^{-i/2},
\end{equation}
by the basic result of asymptotic analysis. From \eqref{eq22} and \eqref{eq23}, we have
\begin{align*}
\frac{f(X+r)}{f(X)}=&\exp\left(\frac{\beta_fr}{2\sqrt{X}}\right)e^{\beta_f\sqrt{X}\left(\sqrt{1+\frac{r}{X}}-1-\frac{r}{2X}\right)}
\left(\sum_{n\ge 0}\frac{\gamma_n(f)}{X^{\frac{n}{2}+\alpha_f}}\right)^{-1}\sum_{n\ge 0}\frac{\gamma_n(f)}{(X+r)^{\frac{n}{2}+\alpha_f}}\\
\sim &\exp\left(\frac{\beta_fr}{2\sqrt{X}}\right)\sum_{k,\ell\ge 0}d_{k,\ell}(f)\frac{r^{k+2\ell}}{X^{3\ell/2+k}}
\sum_{d,i\ge 0}c_{d,i}(f)\frac{r^{d}}{X^{d+i/2}}\\
\sim & \exp\left(\frac{\beta_fr}{2\sqrt{X}}\right)\sum_{j\ge 0}r^j\sum_{\substack{k,\ell,d,i\ge 0\\ k+2\ell+d=j}}
X^{-3\ell/2-k-d-i/2}c_{d,i}(f)d_{k,\ell}(f)\\
\sim & \exp\left(\frac{\beta_fr}{2\sqrt{X}}\right)\sum_{j\ge 0}r^jX^{-j}\sum_{\substack{k,\ell,d,i\ge 0\\ k+2\ell+d=j}}
X^{(\ell-i)/2}c_{d,i}(f)d_{k,\ell}(f).
\end{align*}
Namely,
\begin{align*}
\frac{f(X+r)}{f(X)}
\sim \exp\left(\frac{\beta_fr}{2\sqrt{X}}\right)\sum_{j\ge 0}\left(\frac{r}{X^{3/4}}\right)^j\sum_{n\ge 0}X^{-n/4}
\sum_{\substack{k,d,\ell, i\ge 0\\2(\ell-i)+n=j,~ k+2\ell+d=j}}c_{d,i}(f)d_{k,\ell}(f).
\end{align*}
Which means that we have an asymptotic expansion with respect to the \emph{asymptotic sequence}\footnote{For the definition of asymptotic sequence, see \cite[Section 2.1(v)]{NIST:DLMF}.}
$\{\left(X^{-3/4}r\right)^j\}$ of the form
\begin{equation*}
\frac{f(X+r)}{f(X)}\sim e^{\frac{\beta_fr}{2\sqrt{X}}}\sum_{j\ge 0}\Lambda_{j}(f,X) \left(\frac{r}{X^{3/4}}\right)^j,
\end{equation*}
where
$$\Lambda_{j}(f,X)\sim \sum_{n\ge 0}\lambda_{n,j}(f)X^{-n/4}$$
for some constants $\lambda_{n,j}(f)$ given by
\begin{equation}\label{ldef}
\lambda_{n,j}(f)=\sum_{\substack{k,d,\ell, i\ge 0\\2(\ell-i)+n=j,~ k+2\ell+d=j}}c_{d,i}(f)d_{k,\ell}(f)\in\rb,
\end{equation}
with $c_{d,i}(f)$ and $d_{k,\ell}(f)$ are defined by \eqref{eq22} and \eqref{gke}, respectively.  This completes the proof.
\end{proof}

From the above proposition, we derive the following corollary, which will be used in the proof of Theorem \ref{thm1}.
\begin{corollary}\label{cor21}Let $u\in\rb$ with $u= O(1)$. We have an asymptotic expansion of the form
\begin{align*}
\frac{f(X+u)}{f(X)}&\sim \sum_{d\ge 0}X^{-d/2}\sum_{\substack{k,j,n\ge 0\\ n+2k+3j=2d}}
\frac{\lambda_{n,j}(f)\left({\beta_f}/{2}\right)^{k}}{k !}
u^{k+j}
\end{align*}
as $X\rrw +\infty$.
\end{corollary}
\begin{proof}By Proposition \ref{lem22} we have
\begin{align*}
\frac{f(X+u)}{f(X)}&\sim e^{\frac{\beta_fu}{2\sqrt{X}}}\sum_{j\ge 0}u^jX^{-3j/4}
\sum_{n\ge 0}\lambda_{n,j}(f)X^{-n/4}\\
&\sim \sum_{k,j,n\ge 0}X^{-\frac{n+2k+3j}{4}}\frac{\lambda_{n,j}(f)\left({\beta_f}/{2}\right)^{k}}{k !}
u^{k+j},
\end{align*}
as $X\rrw +\infty$. Further, by noting that $\lambda_{n,j}=0$ for $n\not\equiv j\bmod 2$, we find that
\begin{align*}
\frac{f(X+u)}{f(X)}&\sim \sum_{d\ge 0}X^{-d/2}\sum_{\substack{k,j,n\ge 0\\ n+2k+3j=2d}}
\frac{\lambda_{n,j}(f)\left({\beta_f}/{2}\right)^{k}}{k !}
u^{k+j},
\end{align*}
which completes the proof of the corollary.
\end{proof}

\subsection{Uniform asymptotics of a false theta function}\label{sec22}

In this subsection we investigate the uniform asymptotics of the false theta function \eqref{PTF}. We first deduce the following proposition.
\begin{proposition}\label{lemth}Let $\alpha, z\in\cb$ with $\Re(z)>0$, $\alpha\not\in 2\pi\ri(\zb+1/2)$, $\ell\in\nb_0$ and $N\in\nb$. We have
$$\sum_{n\ge 1}(-1)^{n-1}n^{\ell}e^{-n^2z-n\alpha}=(-1)^{\ell}\sum_{k=0}^{N-1}\frac{(-z)^{k}}{k !}\P_{\alpha}^{2k+\ell}
\left(\frac{1}{1+e^{\alpha}}\right)+z^{N}R_{N}(\ell,\alpha,z),$$
where,
\begin{equation*}
R_{N}(\ell,\alpha,z)=\frac{(-1)^{N+\ell}}{(N-1)!}\int_{0}^1(1-t)^{N-1}\P_{\alpha}^{2N+\ell}
\bigg(\sum_{n\ge 1}(-1)^{n-1}e^{-n^2zt-n\alpha}\bigg)\,dt.
\end{equation*}
\end{proposition}
\begin{proof}Obviously, both sides of the equality we need prove are holomorphic with respect to $\alpha\in\cb\setminus 2\pi\ri(\zb+1/2)$.
Thus we just need to prove the cases of $\Re(\alpha)>0$, then the proof follows from the identity theorem of analytic continuation.
Suppose that $\Re(\alpha)>0$ and denote
$F_{\ell}(z)=\sum_{n\ge 1}(-1)^{n-1}n^{\ell}e^{-n^2z-n\alpha}$. We have for each $k\in\nb_0$,
\begin{align*}
F_{\ell}^{(k)}(0^+)&:=\lim_{z\rrw 0^+}F_{\ell}^{(k)}(z)\\
&=\lim_{z\rrw 0^+}\sum_{n\ge 1}(-1)^{n-1+k}n^{2k+\ell}e^{-n^2z-n\alpha}\\
&=(-1)^k\sum_{n\ge 1}(-1)^{n-1}n^{2k+\ell}e^{-n\alpha}\\
&=(-1)^{k+\ell}\P_{\alpha}^{2k+\ell}\sum_{n\ge 1}(-1)^{n-1}e^{-n\alpha},
\end{align*}
that is
\begin{align*}
F_{\ell}^{(k)}(0^+)=(-1)^{k+\ell}\P_{\alpha}^{2k+\ell}\left(\frac{1}{1+e^{\alpha}}\right).
\end{align*}
Thus from Taylor's theorem we have
\begin{align*}
F_{\ell}(z)&=\sum_{k=0}^{N-1}\frac{1}{k!}F_{\ell}^{(k)}(0^+)(z-0^+)^k+\int_{0}^{z}\frac{F_{\ell}^{(N)}(t)}{(N-1)!}(z-t)^{N-1}\,dt\\
&=(-1)^{\ell}\sum_{k=0}^{N-1}\frac{(-z)^{k}}{k !}\P_{\alpha}^{2k+\ell}\left(\frac{1}{1+e^{\alpha}}\right)+z^{N}R_{N}(\ell,\alpha,z),
\end{align*}
holds for each $N\in\nb$, with
\begin{align*}
R_{N}(\ell,\alpha,z)&=\frac{1}{(N-1)!}\int_{0}^1\P_{u}^{N}\big|_{u=zt}\left(F_{\ell}(u)\right)(1-t)^{N-1}\,dt\\
&=\frac{(-1)^{N}}{(N-1)!}\int_{0}^1(1-t)^{N-1}\sum_{n\ge 1}(-1)^{n-1}n^{2N+\ell}e^{-n^2zt-n\alpha}\,dt\\
&=\frac{(-1)^{N+\ell}}{(N-1)!}\int_{0}^1(1-t)^{N-1}\P_{\alpha}^{2N+\ell}\bigg(\sum_{n\ge 1}(-1)^{n-1}e^{-n^2zt-n\alpha}\bigg)\,dt,
\end{align*}
which completes the proof.
\end{proof}
We next study the error term $R_{N}(\alpha,z)$ of Proposition \ref{lemth}. To estimate $R_{N}(\alpha,z)$ we need the following lemmas. We first need the following well--known Euler transform.
\begin{lemma}\label{lem41}Let $x\in \cb\setminus\{1\}$ and let $h: \zb\rrw \cb$ satisfy $\sum_{n\ge 0}|x^n h(n)|<+\infty$.
Then, we have
$$\sum_{n\ge 0}x^nh(n)=\sum_{r=0}^{K-1}\frac{x^r\Delta^rh(0)}{(1-x)^{r+1}}+\frac{x^{K}}{(1-x)^{K}}\sum_{n\ge 0}x^n
\Delta^{K}h(n),\quad K\in\nb.$$
\end{lemma}

\begin{lemma}\label{lem211}Let $r, J\in\nb_0$ be given and let $\alpha=x+\ri y$ with $x,y\in\rb$ and $|y|\le x$. We have that
$$e^{\alpha}\P_{\alpha}^{J}\left\{\frac{e^{r\alpha}}{(1+e^{\alpha})^{r+1}}\right\}\ll 1,$$
holds uniformly for all $x\ge 0$.
\end{lemma}
\begin{proof}
Since $\Re(\alpha)=x\ge 0$, we have $|e^{-\alpha}|\le 1$ and hence
\begin{align*}
e^{\alpha}\P_{\alpha}^{J}\left\{\frac{e^{r\alpha}}{(1+e^{\alpha})^{r+1}}\right\}&=e^{\alpha}\P_{\alpha}^{J}\left\{\frac{e^{-\alpha}}{(1+e^{-\alpha})^{r+1}}\right\}\\
&=e^{\alpha}\sum_{k=0}^{J}\binom{J}{k}(-1)^{J-k}e^{-\alpha}\P_{\alpha}^{k}\left\{\frac{1}{(1+e^{-\alpha})^{r+1}}\right\}\\
&\ll \sum_{k=0}^{J}\left|\P_{\alpha}^{k}\left\{\frac{1}{(1+e^{-\alpha})^{r+1}}\right\}\right|,
\end{align*}
that is
\begin{align*}
e^{\alpha}\P_{\alpha}^{J}\left\{\frac{e^{r\alpha}}{(1+e^{\alpha})^{r+1}}\right\}\ll \sum_{k=0}^{J}\frac{1}{|1+e^{-\alpha}|^{r+1+k}}.
\end{align*}
Since $|y|\le x$, we have
$$|1+e^{-\alpha}|^2=1+e^{-2x}+2e^{-x}\cos(y)\ge \begin{cases}(1-e^{-x})^2~ &if~ x\ge \pi/2,\\
1+e^{-2x}&if~|x|\le \pi/2,
\end{cases}$$
that is $|1+e^{-\alpha}|\gg 1$. Therefore,
\begin{align*}
e^{\alpha}\P_{\alpha}^{J}\left\{\frac{e^{r\alpha}}{(1+e^{\alpha})^{r+1}}\right\}\ll 1,
\end{align*}
which completes the proof of the lemma.
\end{proof}
We now estimate the error term $R_{N}(\alpha,z)$ of Proposition \ref{lemth}. We prove the following proposition.
\begin{proposition}\label{pro24}Let $R_{N}(\ell, \alpha,z)$ be defined as in Proposition \ref{lemth}, and let $\ell\in\nb_0$ and $N\in\nb$ be fixed.
For $z=x+\ri y$ with $x,y\in\rb$, $x>0$ and $|y|\le x$, we have that
\begin{align*}
e^{bz}R_{N}(\ell, bz,z)\ll 1,
\end{align*}
holds uniformly for all $b\ge 0$.
\end{proposition}
\begin{proof}
First of all, from Proposition \ref{lemth} we have
\begin{align}\label{eq280}
R_{N}(\ell, bz,z)\ll\int_{0}^1\left|\P_{\alpha}^{2N+\ell}\big|_{\alpha=bz}\bigg(\sum_{n\ge 1}(-1)^{n}e^{-n^2zt-n\alpha}\bigg)\right|\,dt.
\end{align}
Let $E_u(n)=e^{-n^2u}$ and let $K\in\nb_0$ be given. Trivially, if $n\gg 1/x$ then
\begin{equation}\label{eq281}
\Delta_n^{K}E_{zt}(n)=(-1)^{K}\sum_{r=0}^{K}\binom{K}{r}(-1)^re^{-(n+r)^2zt}\ll e^{-n^2xt}.
\end{equation}
By the Taylor expansion for $e^x$, we further have for all $|zt|\le 1$,
\begin{align*}
e^{n^2zt}\Delta^{K}E_{zt}(n)&=(-1)^{K}e^{n^2zt}\sum_{r=0}^{K}\binom{K}{r}(-1)^re^{-(n+r)^2zt}\\
&\ll \left|\sum_{r=0}^{K}\binom{K}{r}(-1)^re^{-(2nr+r^2)zt}\right|\\
&=\left|\sum_{r=0}^{K}\binom{K}{r}(-1)^re^{-2nrzt}\left(\sum_{0\le \ell< K/2}\frac{(-r^2zt)^{\ell}}{\ell!}+O(|zt|^{K/2})\right)\right|.
\end{align*}
Further, for $\Re(nzt)\ge 0$ we have
\begin{align*}
e^{n^2zt}\Delta^{K}E_{zt}(n)&\ll \left|\sum_{0\le \ell<K/2}\frac{(-zt)^{\ell}}{\ell!}
\sum_{r=0}^{K}\binom{K}{r}(-1)^rr^{2\ell}e^{-2nrzt}\right|+|zt|^{K/2}\\
&=\left|\sum_{0\le \ell<K/2}\frac{(-zt)^{\ell}}{\ell!}\P_u^{2\ell}\big|_{u=2nzt}\left(1-e^{-u}\right)^K\right|+|zt|^{K/2}.
\end{align*}
Thus if $0\le n\le 1/x$, $0\le t\le 1$ and $|zt|\le 1$ then
\begin{align*}
e^{n^2zt}\Delta^{K}E_{zt}(n)&\ll \sum_{0\le \ell<K/2}|zt|^{\ell}|nzt|^{K-2\ell}+|zt|^{K/2}\\
&\ll |nxt|^{K}+(xt)^{K/2}+(xt)^{K/2}\\
&\ll |nxt|^{K}+(xt)^{K/2}.
\end{align*}
Combining the trivial estimate \eqref{eq280} and above we obtain that for all $n\in\nb_0$,
\begin{equation}\label{eq33}
\Delta^{K}E_{zt}(n)\ll \min\left(1, (xt)^{K/2}+|nxt|^{K}\right)e^{-n^2xt}.
\end{equation}
On the other hand, if $\Re(u)>0$ then Lemma \ref{lem41} implies that,
\begin{align*}
\sum_{n\ge 0}(-1)^{n}e^{-n^2u-n\alpha}=&\sum_{r=0}^{K-1}\frac{(-1)^re^{-r\alpha}\Delta^{r}E_{u}(0)}{(1+e^{-\alpha})^{r+1}}\\
&+\frac{(-1)^Ke^{-K\alpha}}{(1+e^{-\alpha})^{K}}\sum_{n\ge 0}(-1)^ne^{-n\alpha}\Delta^{K}E_{u}(n),
\end{align*}
that is
\begin{align}\label{eq34}
\sum_{n\ge 1}(-1)^{n}e^{-n^2u-n\alpha}=&-\frac{e^{-\alpha}}{1+e^{-\alpha}}+\sum_{r=1}^{K-1}\frac{(-1)^re^{-r\alpha}\Delta^{r}E_{u}(0)}{(1+e^{-\alpha})^{r+1}}\nonumber\\
&+\frac{(-1)^Ke^{-K\alpha}}{(1+e^{-\alpha})^{K}}
\sum_{n\ge 0}(-1)^ne^{-n\alpha}\Delta^{K}E_{u}(n).
\end{align}
Hence by inserting \eqref{eq34} and \eqref{eq33} to \eqref{eq280}, using Lemma \ref{lem211} and product rule for derivative, directly calculating yields
\begin{align*}
e^{bz}R_{N}(\ell, bz,z)\ll&\int_{0}^1\bigg|e^{bz}\P_{\alpha}^{2N+\ell}\big|_{\alpha=bz}\bigg(-\frac{e^{-\alpha}}{1+e^{-\alpha}}+\sum_{r=1}^{K-1}\frac{(-1)^re^{-r\alpha}\Delta^{r}E_{zt}(0)}{(1+e^{-\alpha})^{r+1}}\\
&+\frac{(-1)^Ke^{-K\alpha}}{(1+e^{-\alpha})^{K}}
\sum_{n\ge 0}(-1)^ne^{-n\alpha}\Delta^{K}E_{zt}(n)\bigg)\bigg|\,dt\\
\ll& 1+\int_{0}^1\sum_{0\le j\le 2N+\ell}\left|\sum_{n\ge 0}(-1)^nn^je^{-bnz}
\Delta^{K}E_{zt}(n)\right|\,dt\\
\ll&1+ \int_{0}^1\left|\sum_{n\ge 0}n^{2N+\ell}\min\left(1, (xt)^{K/2}+(nxt)^{K}\right)e^{-n^2xt}\right|\,dt,
\end{align*}
for all $b\ge 0$.  Splitting the inner sum of the above integral into two parts, gives
\begin{align*}
e^{bz}R_{N}(\ell, bz,z)
\ll&1+\int_{0}^1\bigg|\sum_{0\le n\le (xt)^{-2/3}}((xt)^{K/2}+(nxt)^{K})n^{2N+\ell}\\
&\qquad\qquad+\sum_{n\ge (xt)^{-2/3}}n^{2N+\ell}e^{-n^2xt}\bigg|\,dt\\
\ll& 1+\int_{0}^1\left|(xt)^{\frac{K}{3}-\frac{2}{3}(2N+\ell+1)}+1\right|\,dt\ll 1,
\end{align*}
by setting $K=4N+2\ell+2$. This completes the proof.
\end{proof}

From Lemma \ref{lem211}, Proposition \ref{lemth} and Proposition \ref{pro24} we prove the following uniform
asymptotic expansion.
\begin{theorem}\label{thmft}Let $p, \ell\in\nb_0$ be given. Also let $z=x+\ri y$ with $x,y\in\rb$, $x>0$ and $|y|\le x$.  We have
$$\sum_{n\ge 1}(-1)^{n-1}n^{\ell}e^{-n^2z-bnz}=\left((-1)^{\ell}\sum_{0\le k<p}\frac{(-z)^{k}}{k !}
\P_{\alpha}^{2k+\ell}+O(|z|^p)\right)\bigg|_{\alpha=bz}\left\{\frac{1}{1+e^{\alpha}}\right\},$$
as $z\rrw 0$, holding uniformly for all $b\ge 0$.
\end{theorem}
To prove our main theorem, we shall prove the following proposition.
\begin{proposition}\label{pro25} Let $\mu\in\rb$,  $J\in\nb_0$ and $p\ge 1$ be given and $z\in\rb_+$. We have
$$\P_z^JT_{a, b+\mu}(z)= \left(\sum_{0\le k<p}z^{k}P_{k,J}(\mu, a, b, \P_{\alpha})+O\left((1+|b|^{J})z^p\right)\right)
\bigg|_{\alpha=bz}\left\{\frac{1}{1+e^{\alpha}}\right\}, $$
as $z\rrw 0^+$, holding uniformly for all $b\ge 0$. Here
$$P_{k,J}(\mu,a, b, \P_{\alpha})=\frac{1}{k!}\sum_{\substack{0\le r\le J\\ 0\le s\le k+J-r}}
\binom{J}{r}\binom{J-r+k}{s}\mu^{k+J-s-r}(-a)^sb^r\P_{\alpha}^{s+k+J}.$$

\end{proposition}
\begin{proof}
By Taylor expansion for $e^{x}$, it is easy to see that
\begin{align*}
T_{a, b+\mu}(z)=\sum_{n\ge 1}(-1)^{n-1}e^{-an^2z-bnz-n\mu z}=\sum_{\ell\ge 0}\frac{(-\mu z)^{\ell}}{\ell !}\sum_{n\ge 1}(-1)^{n-1}n^{\ell}e^{-(an^2+bn)z}.
\end{align*}
Take the $J$-th order derivative of both sides above, and using the product rule we obtain
\begin{align}\label{equ2J}
\P_z^JT_{a, b+\mu}(z)&=\P_z^J\left(\sum_{\ell\ge 0}\frac{(-\mu)^{\ell}}{\ell !}z^{\ell}
\sum_{n\ge 1}(-1)^{n-1}n^{\ell}e^{-(an^2+bn)z}\right)\nonumber\\
&=\sum_{\ell\ge 0}\frac{(-\mu)^{\ell}}{\ell !}\sum_{v=0}^{ J}\binom{J}{v}(-1)^{v}(-\ell)_{v}z^{\ell-v}\P_z^{J-v}
\sum_{n\ge 1}(-1)^{n-1}n^{\ell}e^{-(an^2+bn)z}.
\end{align}
Here $(x)_{v}=\prod_{0\le j<v}(x+j)$ is the Pochhammer symbol.  On the other hand, for each nonnegative integer $N$,
\begin{align}\label{equ2J1}
\P_z^{N}\sum_{n\ge 1}(-1)^{n-1}&n^{\ell}e^{-(an^2+bn)z}\nonumber\\
&=\sum_{n\ge 1}(-1)^{n-1+N}n^{\ell}(an^2+bn)^{N}e^{-(an^2+bn)z}\nonumber\\
&=(-1)^N\sum_{r=0}^N\binom{N}{r}a^{N-r}b^r\sum_{n\ge 1}(-1)^{n-1}n^{\ell+2N-r}e^{-(n^2+a^{-1}bn)az}.
\end{align}
Inserting \eqref{equ2J1} into \eqref{equ2J} implies that
\begin{align*}
\P_z^JT_{b+\mu}(z)
=&\sum_{\ell\ge 0}(-\mu)^{\ell}\sum_{v=0}^{ J}\sum_{r=0}^{J-v}\frac{z^{\ell-v}}{\ell !}
\frac{(-1)^J(-\ell)_{v}a^{J-v-r}b^rJ!}{(J-v-r)!v!r!}\\
&\qquad\qquad\qquad\times\sum_{n\ge 1}(-1)^{n-1}n^{\ell+2J-2v-r}e^{-(n^2+a^{-1}bn)az}\\
=&(-1)^J\sum_{v=0}^{J}\sum_{r=0}^{J-v}\frac{J!a^{J-v-r}b^r}{(J-v-r)!v!r!}
\sum_{\ell\ge 0}\frac{\mu^{\ell+v}(-z)^{\ell}}{\ell !}\\
&\qquad\qquad\qquad\times\sum_{n\ge 1}(-1)^{n-1}n^{\ell+2J-v-r}e^{-(n^2+a^{-1}bn)az}.
\end{align*}
Using Theorem \ref{thmft} yields for any given integer $p\ge 0$
\begin{align*}
\P_z^JT_{a, b+\mu}(z)=&(-1)^J\sum_{v=0}^{J}\sum_{r=0}^{J-v}\frac{J!a^{J-v-r}b^r}{(J-v-r)!v!r!}
\sum_{\ell\ge 0}\frac{\mu^{\ell+v}(-z)^{\ell}}{\ell !}(-1)^{\ell+2J-v-r}\\
&\times\left(\sum_{0\le n<p}\frac{(-az)^n}{n!}\P_{\alpha}^{2n+\ell+2J-v-r}+O(|az|^p)\right)
\bigg|_{\alpha=a^{-1}baz}\left\{\frac{1}{1+e^{\alpha}}\right\}\\
=& \left(\sum_{0\le k<p}z^{k}P_{k,J}(\mu, a, b, \P_{\alpha})+O\left((1+|b|^{J})z^p\right)\right)
\bigg|_{\alpha=bz}\left\{\frac{1}{1+e^{\alpha}}\right\}.
\end{align*}
Here,
\begin{align*}
P_{k,J}(\mu, a, b, \P_{\alpha})&=\sum_{\substack{n,\ell\ge 0\\ n+\ell=k}}(-1)^{J+n+\ell}
\sum_{v=0}^{J}\sum_{r=0}^{J-v}\frac{J!a^{J-v-r+n}b^r\mu^{\ell+v}(-1)^{\ell-v-r}}{(J-v-r)!v!r!\ell ! n!}\P_{\alpha}^{2n+\ell+2J-v-r}\\
&=\sum_{s, r\ge 0}\bigg(\sum_{\substack{v, n,\ell\ge 0\\ n+\ell=k\\ r+v\le J,~ J-v-r+n=s}}
\frac{(-1)^{n+J-v-r}J!}{(J-v-r)!v!r!\ell ! n!}\bigg)\mu^{k+J-s-r}a^sb^r\P_{\alpha}^{s+k+J}\\
&=\sum_{s, r\ge 0}\frac{1}{k!}\sum_{\substack{v, J-r-v, n\ge 0\\ J-r-v+n=s}}
\binom{J}{r}\binom{J-r}{J-r-v}\binom{k}{n}\mu^{k+J-s-r}(-a)^sb^r\P_{\alpha}^{s+k+J}\\
&=\frac{1}{k!}\sum_{\substack{0\le r\le J\\ 0\le s\le k+J-r}}\binom{J}{r}\binom{J-r+k}{s}\mu^{k+J-s-r}(-a)^sb^r\P_{\alpha}^{s+k+J},
\end{align*}
by using the well known Chu--Vandermonde identity
$$\sum_{\substack{m,n\ge 0\\ m+n=s}}\binom{M}{m}\binom{N}{n}=\binom{M+N}{s},$$
for $s\in\nb_0$ and $M,N\in\cb$, which completes the proof.
\end{proof}

\section{Asymptotic expansion for the sum of \texorpdfstring{$f(X)$}{Lg} over values of certain quadratic polynomials}\label{sec3}

In this section, we use the fundamental results of the previous section to prove the main results of this paper.
We first prove Theorem \ref{mth}.
\subsection{The proof of Theorem \ref{mth}}\label{sec31}

We first prove a uniform asymptotic expansion for $S_f(a,b;X)$ in terms of the false theta function $T_{a,b}(z)$ defined
in Subsection \ref{sec22}. To prove it we need the following.

\begin{lemma}\label{lemm301}Let $a\in\rb_+, p\in\nb_0$ be given, $b\in(\rb_+\cup \{0\})$ and $\varepsilon\in\rb_+$. As $\varepsilon\rrw 0^+$,
$$\sum_{n\ge 1}(an^2+bn)^pe^{-\varepsilon (an^2+bn)}\ll
\left(\varepsilon^{-p-1/2}+b^p\right)e^{-\varepsilon b}.$$
\end{lemma}
\begin{proof}
For $0\le b\le \varepsilon^{-1/2}$,
\begin{align}\label{eq002}
\sum_{n\ge 1}(an^2+bn)^pe^{-\varepsilon (an^2+bn)}&\ll \sum_{1\le n\le \varepsilon^{-1/2}}(n^{2p}+b^pn^p)+\sum_{n\ge \varepsilon^{-1/2}}n^{2p}e^{-an^2\varepsilon}\nonumber\\
&\ll \varepsilon^{-p-1/2}+b^{p}\varepsilon^{-p/2-1/2}+\sum_{n\in \zb}n^{2p}e^{-an^2\varepsilon}\nonumber\\
&\ll \varepsilon^{-p-1/2}e^{-\varepsilon b},
\end{align}
by using the fact that for each $p\in\nb_0$,
\begin{equation}\label{eq000}
\sum_{n\in \zb}n^{2p}e^{-an^2\varepsilon}\ll \varepsilon^{-p-1/2},
\end{equation}
as $\varepsilon\rrw 0^+$. For $b\ge \varepsilon^{-1/2}$,
\begin{align*}
\sum_{n\ge 1}(an^2+bn)^pe^{-\varepsilon (an^2+bn)}&\ll \sum_{n\ge 1}(n^{2p}+b^pn^p)e^{-\varepsilon (an^2+bn)}\\
&\ll e^{-\varepsilon b}\left(\sum_{n\ge 1}n^{2p}e^{-an^2\varepsilon}+b^p\sum_{n\ge 1}n^pe^{-\varepsilon (an^2+bn)}\right)\\
&\ll e^{-\varepsilon b}\left(\sum_{n\in \zb}n^{2p}e^{-an^2\varepsilon}+b^p+b^p\sum_{n\ge 1}n^pe^{-\varepsilon bn}\right).
\end{align*}
Therefore, by noting that
\begin{align*}
\sum_{n\ge 1}n^pe^{-\varepsilon bn}&=(-1)^p\frac{\,d}{\,d \alpha^p}\bigg|_{\alpha=b \varepsilon}\sum_{n\ge 1}e^{-\alpha n}\\
&=(-1)^p\frac{\,d}{\,d \alpha^p}\bigg|_{\alpha=b \varepsilon}\frac{1}{e^{\alpha}-1}\\
&\ll \max\left(1,\frac{1}{(b\varepsilon)^{p+1}}\right),
\end{align*}
and \eqref{eq000} we have
\begin{align}\label{eq001}
\sum_{n\ge 1}(an^2+bn)^pe^{-\varepsilon (an^2+bn)}&\ll e^{-\varepsilon b}\left(\varepsilon^{-p-1/2}+b^p+\varepsilon^{-p-1}/b\right)\nonumber\\
&\ll \left(\varepsilon^{-p-1/2}+b^p\right)e^{-\varepsilon b}.
\end{align}
Combining \eqref{eq002} and \eqref{eq001} finishes the proof.
\end{proof}

We next prove the following proposition.
\begin{proposition}\label{pro31}Let $a\in\rb_+$, $p\in \nb$ be given and let $X, b\in(\rb_+\cup \{0\})$. We have
$$
\frac{S_{f}(a,b;X)}{f(X)}=\sum_{0\le \ell<p}X^{-3\ell/4}\Lambda_{\ell}(f,X)\P_z^{\ell}\Big|_{z=\frac{\beta_f}{2\sqrt{X}}}T_{a, b}(z)
+O\left(\frac{1+b^{p}}{X^{3p/4}}e^{-\frac{b\beta_f}{2\sqrt{X}}}\right),
$$
as $ X\rrw +\infty$, for all $b=o(X^{3/4})$.
\end{proposition}

\begin{proof}Let $\varepsilon>0$ be sufficiently small be given and let $0\le b\le \varepsilon X^{3/4}/4$. We first split the sum $S_f(a,b;X)$
into two parts as follows:
\begin{align*}
\frac{S_f(a,b;X)}{f(X)}&=\left(\sum_{\substack{n\ge 1\\ an^2+bn\le \varepsilon X^{3/4}}}
+\sum_{\substack{n\ge 1\\ \varepsilon X^{3/4}< an^2+bn\le X}}\right)(-1)^{n-1}\frac{f(X-(an^2+bn))}{f(X)}\\
&=:M+E.
\end{align*}
For the sum $E$, we estimate that
\begin{align*}
E&\ll \sum_{\substack{n\ge 1\\ \varepsilon X^{3/4}< an^2+bn\le X}}\frac{f(X-(an^2+bn))}{f(X)} \\
&\ll \sqrt{X} \frac{f(X-\varepsilon X^{3/4})}{f(X-a-b)}\nonumber\\
 &\ll \sqrt{X}e^{\beta_f\sqrt{X-\varepsilon X^{3/4}}-\beta_f\sqrt{X}},
\end{align*}
that is
\begin{align}\label{equerr}
E\ll  \sqrt{X}e^{-\frac{\beta_f}{2\sqrt{X}}\varepsilon X^{3/4}}
\ll X^{-A}e^{-\frac{\beta_fb}{2\sqrt{X}}},
\end{align}
holds for any given $A>0$, by using that $0\le b\le \varepsilon X^{3/4}/4$. For the sum $M$, applying Proposition \ref{lem22}
implies that,
\begin{align*}
M=&\sum_{\substack{n\ge 1\\ an^2+bn\le \varepsilon X^{3/4}}}(-1)^{n-1}\bigg(\sum_{0\le j<p}X^{-3j/4}\Lambda_{j}(f,X)\P_y^j\\
&\qquad\qquad\qquad+O\left(\left|\frac{an^2+bn}{X^{3/4}}\right|^p\right)\bigg)\bigg|_{y=\frac{\beta_f}{2\sqrt{X}}} e^{-y (an^2+bn)}\\
=& \sum_{0\le j<p}\frac{\Lambda_{j}(f,X)}{X^{3j/4}}\P_y^{j}\big|_{y=\frac{\beta_f}{2\sqrt{X}}}
\sum_{\substack{n\ge 1\\ an^2+bn\le \varepsilon X^{3/4}}}(-1)^{n-1}e^{-y (an^2+bn)}\\
&\qquad\qquad\qquad+O\left(\sum_{\substack{n\ge 1\\ an^2+bn\le \varepsilon X^{3/4}}}\frac{(an^2+bn)^{p}}{X^{3p/4}}e^{-\frac{\beta_f}{2\sqrt{X}} (an^2+bn)}\right).
\end{align*}
Since for each $j\in\nb$,
\begin{align*}
\P_y^{j}\big|_{y=\frac{\beta_f}{2\sqrt{X}}}\Bigg(&\sum_{\substack{n\ge 1\\ an^2+bn\ge \varepsilon X^{3/4}}}(-1)^{n-1}e^{-y (n^2+2bn)}\Bigg)\\
&\ll \sum_{\substack{n\ge 1\\ an^2+bn\ge  \varepsilon X^{3/4}}}(an^2+bn)^{j}e^{-\frac{\beta_f}{2\sqrt{X}}(an^2+bn)}\\
&\ll e^{-\frac{3\varepsilon\beta_f}{8}X^{1/4}}\sum_{n\ge 1}(an^2+bn)^{j}e^{-\frac{\beta_f}{8\sqrt{X}}(an^2+bn)}\\
&\ll e^{-\frac{3\varepsilon\beta_f}{8}X^{1/4}}\left(X^{j/2+1/4}+b^j\right)e^{-\frac{\beta_fb}{8\sqrt{X}}},
\end{align*}
by Lemma \ref{lemm301} and using that $0\le b\le \varepsilon X^{3/4}/4$, we have for any given $A>0$,
\begin{equation}\label{equerr1}
\P_y^{j}\big|_{y=\frac{\beta_f}{2\sqrt{X}}}\Bigg(\sum_{\substack{n\ge 1\\ an^2+bn\ge \varepsilon X^{3/4}}}(-1)^{n-1}e^{-y (n^2+2bn)}\Bigg)
\ll X^{-A}e^{-\frac{\beta_fb}{2\sqrt{X}}}.
\end{equation}
By Lemma \ref{lemm301},
\begin{equation}\label{equerr2}
\sum_{\substack{n\ge 1\\ an^2+bn\le \varepsilon X^{3/4}}}\frac{(an^2+bn)^p}{X^{3p/4}}e^{-\frac{\beta_f}{2\sqrt{X}} (an^2+bn)}
\ll \left(\frac{X^{p/2+1/4}+b^p}{X^{3p/4}}\right)e^{-\frac{\beta_fb}{2\sqrt{X}}}.
\end{equation}
Combining \eqref{equerr}--\eqref{equerr2}, changing $p$ to $8p$ and noting that Proposition \ref{pro25} gives
$$\P_z^{\ell}\big|_{z=\frac{\beta_f}{2\sqrt{X}}}T_{a, b}(z)\ll (1+b^{\ell})e^{-\frac{b\beta_f}{2\sqrt{X}}},$$
for each $\ell\in\nb$, we have
\begin{align*}
\frac{S_{f}(a,b;X)}{f(X)}&=\sum_{0\le \ell<8p}X^{-3\ell/4}\Lambda_{\ell}(f,X)\P_z^{\ell}\Big|_{z=\frac{\beta_f}{2\sqrt{X}}}T_{a, b}(z)
+O\left(\frac{X^{4p+1/4}+b^{8p}}{X^{6p}}e^{-\frac{b\beta_f}{2\sqrt{X}}}\right)\\
&=\sum_{0\le \ell<p}X^{-3\ell/4}\Lambda_{\ell}(f,X)\P_z^{\ell}\Big|_{z=\frac{\beta_f}{2\sqrt{X}}}T_{a, b}(z)
+O\left(\frac{1+b^{p}}{X^{3p/4}}e^{-\frac{b\beta_f}{2\sqrt{X}}}\right),
\end{align*}
which completes the proof.
\end{proof}

\begin{proof}[Proof of Theorem \ref{mth}]From Proposition \ref{pro31}, Proposition \ref{pro25} and Proposition \ref{lem22} we find that
\begin{align*}
\frac{S_{f}(a,b+\mu;X)}{f(X)}=&\sum_{0\le j<2p}X^{-3j/4}\Lambda_{j}(f,X)\P_z^{j}\Big|_{z=\frac{\beta_f}{2\sqrt{X}}}T_{a, b+\mu}(z)\\
&\qquad\qquad\qquad+O\left(\frac{1+|b+\mu|^{2p}}{X^{3p/2}}e^{-\frac{(b+\mu)\beta_f}{2\sqrt{X}}}\right)\\
=&\sum_{\substack{0\le j<2p\\ 0\le k,n<4p}}\frac{\lambda_{n,j}(f)}{X^{(n+3j)/4}}\left(\frac{\beta_f}{ 2\sqrt{X}}\right)^{k}\\
&\qquad \times P_{k,j}(\mu, a, b, \P_{\alpha})\big|_{\alpha=\frac{b\beta_f}{2\sqrt{X}}}\left\{\frac{1}{1+e^{\alpha}}\right\}+E_{f,p}(a,b+\mu;X)\\
=&:I_{f,p}(a,b+\mu;X)+E_{f,p}(a,b+\mu;X).
\end{align*}
Note the following bounds
$$\Lambda_j(f, X)\ll 1$$
and
$$\P_z^{j}\Big|_{z=\frac{\beta_f}{2\sqrt{X}}}T_{a, b+\mu}(z)\ll 1+|b|^J,$$
for all $j\in\nb_0$. Which follow from Proposition \ref{lem22} and  Proposition \ref{pro25}. Then using the estimates in Proposition \ref{lem22} and  Proposition \ref{pro25} yields
\begin{align*}
e^{\frac{(b+\mu)\beta_f}{2\sqrt{X}}}E_{f,p}(a,b+\mu;X)\ll & \frac{1+|b|^{2p}}{X^{3p/2}}+\sum_{0\le j<2p}\left(\frac{1+|b|^j}{X^{p+3j/4}}+\frac{1+|b|^j}{X^{2p+3j/4}}\right)
\end{align*}
Since $b=o(X^{3/4})$, we have
 \begin{align*}
E_{f,p}(a,b+\mu;X)\ll  \left(\frac{1+|X^{-1/2}b|}{X}\right)^pe^{-\frac{(b+\mu)\beta_f}{2\sqrt{X}}}.
\end{align*}
Moreover,
\begin{align*}
I_{f,p}(a,&b+\mu;X)\\
&=\sum_{d\ge 0}\frac{1}{X^{d/4}}\sum_{\substack{0\le j<2p\\ 0\le k,n<4p\\ n+2k+3j=d}}
\lambda_{n,j}(f)\left(\beta_f/2\right)^{k}P_{k,j}(\mu, a, b, \P_{\alpha})\big|_{\alpha=\frac{b\beta_f}{2\sqrt{X}}}
\left\{\frac{1}{1+e^{\alpha}}\right\}\\
&=\sum_{d\ge 0}\frac{1}{X^{d/4}}\sum_{\substack{0\le j<2p\\ 0\le k,n<4p\\ n+2k+3j=d}}\sum_{\substack{0\le r\le j\\ 0\le s\le k+j-r}}\frac{\lambda_{n,j}(f)\left(\beta_f/2\right)^{k}\binom{j}{r}\binom{j-r+k}{s}}{k!}\\
&\qquad\qquad\qquad\qquad\qquad\qquad \times
\mu^{k+j-s-r}(-a)^sb^r\P_{\alpha}^{s+k+j}\big|_{\alpha=\frac{b\beta_f}{2\sqrt{X}}}\left\{\frac{1}{1+e^{\alpha}}\right\}\\
&=\sum_{d\ge 0}\frac{1}{X^{d/4}}\sum_{r,\ell,s\ge 0}C_{r,\ell,s}^*(d; f)a^sb^r\mu^{\ell}
\P_{\alpha}^{r+\ell+2s}\big|_{\alpha=\frac{b\beta_f}{2\sqrt{X}}}\left\{\frac{1}{1+e^{\alpha}}\right\},
\end{align*}
with
\begin{align*}
C_{r,\ell,s}^*(d; f)&=\sum_{0\le k,n<4p}\sum_{\substack{2p>j\ge r\\ k+j=s+r+\ell,~ n+2k+3j=d}}(-1)^{s}
\frac{\lambda_{n,j}(f)\left(\beta_f/2\right)^{k}\binom{j}{r}\binom{j-r+k}{s}}{k!}.
\end{align*}
From the above, and noting that $\lambda_{n,j}(f)=0$ if $n\not\equiv j\bmod 2$ (by Proposition \ref{lem22}) it is not difficult
to see that $C_{r,\ell,s}^*(2d+1;f)\equiv 0$. Furthermore, if $0\le d\le p$ then it is clear that
\begin{align*}
C_{r,\ell,s}^*(2d; f)&=(-1)^{s}\sum_{\substack{n,k\ge 0, j\ge r\\ k+j=s+r+\ell,~ n+2k+3j=2d}}\frac{\lambda_{n,j}(f)
\left(\beta_f/2\right)^{k}\binom{j}{r}\binom{j-r+k}{s}}{k!}\\
&=(-1)^{s}\sum_{\substack{n,k, j\ge 0\\ k+j=s+\ell,~ n+2k+3j+3r=2d}}\frac{\lambda_{n,j+r}(f)\left(\beta_f/2\right)^{k}
\binom{j+r}{r}\binom{j+k}{s}}{k!}\\
&=:C_{r,\ell,s}(d; f).
\end{align*}
We further have
\begin{equation}\label{eqccc}
(-1)^{s}C_{r,\ell,s}(d; f)
=\binom{s+\ell}{s}\sum_{\substack{n+j+2s+2\ell+3r=2d\\ 0\le j\le s+\ell, n\ge 0}}\frac{\left(\beta_f/2\right)^{s+\ell-j}}{(s+\ell-j) !}
\binom{j+r}{r}\lambda_{n,j+r}(f),
\end{equation}
with $\lambda_{n,j}(f)$ be given in Proposition \ref{lem22}. It is also clear that
$$
C_{r,\ell,s}^*(2d; f)\ll \sum_{\substack{n+j+2s+2\ell+3r=2d\\ 0\le j\le s+\ell, n\ge 0}}1
\ll \begin{cases} 1\quad & if~2s+2\ell+3r\le 2d,\\
0 &if~2s+2\ell+3r> 2d,
\end{cases}
$$
and hence
\begin{align*}
I_{f,p}(a,b+\mu;X)=\Bigg(\sum_{0\le d<p}X^{-d/2}&\cL_{f,d}(\mu,a,b,\P_{\alpha})
\\
&+O\left(\frac{1+|b|^{2p/3}}{X^{p/2}}\right)\Bigg)\bigg|_{\alpha=\frac{b\beta_f}{2\sqrt{X}}}\left\{\frac{1}{1+e^{\alpha}}\right\},
\end{align*}
with
$$\cL_{f,d}(\mu,a,b,\P_{\alpha})
=\sum_{\substack{r,\ell,s\ge 0\\ 2s+2\ell+3r\le 2d}}C_{r,\ell,s}(d; f)a^sb^r\mu^{\ell}\P_{\alpha}^{r+\ell+2s}.$$
Now combining the estimate for $E_{f,p}(a,b+\mu;X)$ and then changing $p$ to $3p$ completes the proof of Theorem \ref{mth}.
\end{proof}
\subsection{The proof of Theorem \ref{thm1}}\label{sec32}

In this subsection we prove Theorem \ref{thm1}. We first prove the following proposition:
\begin{proposition}\label{pro32}Let $X, b\in\rb_+$ with $b\le X$. If $X/3<b\le X$ then
$$S_{f}(a,b; X)=f(X-a-b)+O(|f(X-4a-2b)|).$$
If $X-b\rrw+\infty$ then
$$\frac{S_{f}(a,b; X)}{f(X-a-b)}=1+O\left(b^{-1}Xe^{-\frac{b\beta_f}{2\sqrt{X}}}\right).$$
\end{proposition}
\begin{proof} From the definition
\begin{align*}
S_{f}(a,b; X)=\sum_{\substack{n\ge 1\\ an^2+bn\le X}}(-1)^{n-1}f\left(X-(an^2+bn)\right)
\end{align*}
we have
\begin{align*}
|S_{f}(a,b; X)-f(X-a-b)|\le |f(X-4a-2b)|+\sum_{\substack{n\ge 3\\ an^2+bn\le X}}|f(X-(an^2+bn))|.
\end{align*}
If $X<3b$ then $X-(an^2+bn)<0$ when $n\ge 3$, and this yields
\begin{align*}
S_{f}(a,b; X)=f(X-a-b)+O(|f(X-4a-2b)|).
\end{align*}
If $X-b\rrw \infty$ then
\begin{align*}
\frac{S_f(a,b;X)}{f(X-(a+b))}-1
&=\sum_{\substack{n\ge 2\\ an^2+bn\le X}}(-1)^{n-1}\frac{f\left(X-(an^2+bn)\right)}{f(X-(a+b))}\\
&\ll\sum_{\substack{n\ge 2\\ an^2+bn\le X}}\frac{f\left(X-(4a+2b)\right)}{f(X-(a+b))}.
\end{align*}
Thus for $2b+4a\le X$,
\begin{align*}
\frac{S_f(a,b;X)}{f(X-(a+b))}-1\ll \frac{X}{b}e^{\beta_f\left(\sqrt{X-4a-2b}-\sqrt{X-a-b}\right)}
\ll b^{-1}Xe^{-\frac{b\beta_f}{2\sqrt{X}}},
\end{align*}
and for $2b+4a>X$ and,
\begin{align*}
\frac{S_f(a,b;X)}{f(X-(a+b))}-1=0\ll b^{-1}Xe^{-\frac{b\beta_f}{2\sqrt{X}}},
\end{align*}
which completes the proof of the Proposition \ref{pro32}.
\end{proof}
\begin{proof}[Proof of Theorem \ref{thm1}] We first let $L(x)$ be a real function satisfying $\lim\limits_{x\rrw+\infty}L(x)=+\infty$. By Proposition \ref{pro32} and Corollary \ref{cor21}, if $X, b\in\rb_+$ such
that $L(X)X^{1/2}\log X\le b\le X-L(X)$, then we have for each $p\in\nb$,
\begin{align*}
\frac{S_f(a,b+\mu;X)}{f(X-b)}=&\frac{S_f(a,b+\mu,X)}{f(X-b-a-\mu)}\frac{f(X-b-a-\mu)}{f(X-b)}\\
=&\left(1+O\left(\frac{X}{b+\mu}e^{-\frac{(b+\mu)\beta_f}{2\sqrt{X}}}\right)\right)\Bigg(O\left(\frac{1}{(X-b)^{p}}\right)\\
&\qquad+\sum_{0\le d<2p}\frac{\sum_{\substack{k,j,n\ge 0\\ n+2k+3j=2d}}\frac{(-1)^{j+k}\lambda_{n,j}(f)\left({\beta_f}/{2}\right)^{k}
(a+\mu)^{k+j}}{k !}}{(X-b)^{d/2}}\Bigg).
\end{align*}
Notice that for any given $A>0$,
$$\frac{X}{b+\mu}e^{-\frac{(b+\mu)\beta_f}{2\sqrt{X}}}\ll b^{-1}X^{1-\frac{\beta_f}{2}L(X)}\ll \frac{1}{(X-b)^{A}},$$
then we have
\begin{align*}
\frac{S_f(a,b+\mu;X)}{f(X-b)}&\sim \sum_{d\ge 0}\frac{1}{(X-b)^{\frac{d}{2}}}\sum_{\substack{k,j,n\ge 0\\ n+2k+3j=2d}}\frac{(-1)^{j+k}
\lambda_{n,j}(f)\left({\beta_f}/{2}\right)^{k}}{k !}(a+\mu)^{k+j}\\
&\sim \sum_{d\ge 0}\frac{1}{(X-b)^{\frac{d}{2}}}\sum_{\substack{\ell, s, k,j,n\ge 0\\ n+2k+3j=2d, \ell+s=k+j}}\frac{(-1)^{j+k}\binom{k+j}{s}
\lambda_{n,j}(f)\left(\frac{\beta_f}{2}\right)^{k}}{k !}\mu^{\ell}a^{s},
\end{align*}
This means that we have an asymptotic expansion of form
\begin{align*}
\frac{S_f(a,b+\mu;X)}{f(X-b)}
\sim \sum_{d\ge 0}\frac{1}{(X-b)^{d/2}}\sum_{\substack{\ell,s\ge 0\\ \ell+s\le d}}C_{\ell,s}(d;f)\mu^{\ell}a^{s},
\end{align*}
as $(X-b)\rrw+\infty$, where
\begin{equation}\label{eqccc1}
(-1)^{\ell+s}C_{\ell,s}(d;f)=\binom{\ell+s}{s}\sum_{\substack{n+j+2(\ell+s)=2d\\ 0\le j\le \ell+s, n\ge 0}}
\frac{\left({\beta_f}/{2}\right)^{\ell+s-j}}{(\ell+s-j) !}\lambda_{n,j}(f).
\end{equation}
This completes the proof of Theorem \ref{thm1}.
\end{proof}
\section{Proofs of Theorem \ref{propm} and \ref{propm'}}\label{sec4}

In this section we prove Theorem \ref{propm} and Theorem \ref{propm'}. We shall use Theorem \ref{mth} and Theorem \ref{thm1}
to prove Theorem \ref{propm} in Subsection \ref{sec41}. We prove Theorem \ref{propm'} in Subsection \ref{sec42}.

\subsection{The proof of Theorem \ref{propm}}\label{sec41}
\subsubsection{Some of the first exact values of coefficients $C_{r,\ell,s}(d;f)$ and  $C_{\ell,s}(d;f)$}\label{sec411}
We begin with the following lemma.
\begin{lemma}\label{lem32}Let $\ell, J\in\nb_0$. We have:
\begin{align*}
\Delta_{u}^J\big|_{u=0}u^{\ell}=\begin{cases}~\qquad 0&\quad 0\le \ell<J,\\
\qquad J! &\qquad \ell=J,\\
 J(J+1)!/2 &\qquad \ell=J+1.
\end{cases}
\end{align*}
\end{lemma}
\begin{proof}
For $J,\ell\in\nb_0$ note that
\begin{align*}
\Delta_{u}^J\big|_{u=0}u^{\ell}&=(-1)^J\sum_{j=0}^J(-1)^j\binom{J}{j}j^{\ell}=(-1)^{\ell}\P_x^{\ell}\Big|_{x=0}\left((-1)^J\sum_{j=0}^J(-1)^j\binom{J}{j}e^{-jx}\right)\\
&=(-1)^{\ell+J}\P_x^{\ell}\Big|_{x=0}\left(1-e^{-x}\right)^{J}=(-1)^{\ell+J}\P_x^{\ell}\Big|_{x=0}\left(x^{J}
-({J}/{2})x^{J+1}+\dots\right).
\end{align*}
This finished the proof of this lemma.
\end{proof}
We now give some of the first exact values of coefficients $C_{r,\ell,s}(d;f)$ and  $C_{\ell,s}(d;f)$ in Theorem \ref{mth} and
Theorem \ref{thm1}, respectively, which will be used in the proof of Theorem \ref{propm}. Using
Proposition \ref{lem22} in Theorem \ref{mth}, we have for each $J\in\nb_0$,
\begin{equation}\label{eqc1}
C_{0,J,0}(J;f)=\frac{\beta_f^J\lambda_{0,0}(f)}{2^JJ!}=\frac{\beta_f^J}{2^JJ!},
\end{equation}
\begin{align}\label{eqc2}
C_{0,J,0}(J+1;f)&=\frac{\beta_f^J\lambda_{2,0}(f)}{2!J!}+\frac{\beta_f^{J-1}\lambda_{1,1}(f)}{2^{J-1}\Gamma(J)}
+\frac{\beta_f^{J-2}\lambda_{0,2}(f)}{2^{J-2}\Gamma(J-1)}\nonumber\\
&=-\frac{\beta_f^{J-1}\alpha_f}{2^{J-1}\Gamma(J)}-\frac{\beta_f^{J-2}\beta_f}{2^{J+1}\Gamma(J-1)},
\end{align}
\begin{equation}\label{eqc3}
-C_{0,J,1}(J+1;f)=\binom{J+1}{J}\frac{\beta_f^{J+1}\lambda_{0,0}(f)}{2^{J+1}(J+1)!}=\frac{\beta_f^{J+1}}{2^{J+1}J!},
\end{equation}
\begin{align}\label{eqc4}
C_{1,J,0}(J+2;f)&=\frac{\beta_f^{J-1}\lambda_{0,2}(f)}{2^{J-1}\Gamma(J)}\binom{2}{1}
+\frac{\beta_f^J\lambda_{1,1}(f)}{2^JJ!}=-\frac{\beta_f^{J}}{2^{J+1}\Gamma(J)}-\frac{\alpha_f\beta_f^J}{2^JJ!}
\end{align}
and
\begin{equation}\label{eqc5}
C_{2,J,0}(J+3;f)=\frac{\beta_f^J\lambda_{0,2}(f)}{2^JJ!}=-\frac{\beta_f^{J+1}}{2^{J+3}J!}.
\end{equation}
Here $\Gamma(x)$ is the Euler gamma function defined by
$$\frac{1}{\Gamma(x)}=x\prod_{k\ge 1}\left(1-\frac{x}{k}\right)\left(1+\frac{1}{k}\right)^{x},$$
for all $x\in\cb$. From Theorem \ref{thm1} and Proposition \ref{lem22}, we have for each $J\in\nb_0$,
\begin{equation}\label{eqc11}
(-1)^JC_{J,0}(J;f)=\frac{\beta_f^J\lambda_{0,0}(f)}{2^JJ!}=\frac{\beta_f^J}{2^JJ!},
\end{equation}
\begin{align}\label{eqc21}
(-1)^JC_{J,0}(J+1;f)&=\frac{\beta_f^J\lambda_{2,0}(f)}{2!J!}+\frac{\beta_f^{J-1}\lambda_{1,1}(f)}{2^{J-1}\Gamma(J)}
+\frac{\beta_f^{J-2}\lambda_{0,2}(f)}{2^{J-2}\Gamma(J-1)}\nonumber\\
&=-\frac{\beta_f^{J-1}\alpha_f}{2^{J-1}\Gamma(J)}-\frac{\beta_f^{J-1}}{2^{J+1}\Gamma(J-1)}
\end{align}
and
\begin{equation}\label{eqc31}
(-1)^{J+1}C_{J,1}(J+1;f)=\binom{J+1}{J}\frac{\beta_f^{J+1}\lambda_{0,0}(f)}{2^{J+1}(J+1)!}=\frac{\beta_f^{J+1}}{2^{J+1}J!}.
\end{equation}

\subsubsection{The cases of $0\le b\le \sqrt{X}(\log X)^2$}\label{sec412}

We first compute some special values of $\Delta_{u}^J\big|_{u=0}\cL_{f,d}(\mu u, a, b,\P_{\alpha}), J\in\nb_0$ in Theorem \ref{mth} by using
Lemma \ref{lem32}. For each nonnegative integer $d<J$, we have
\begin{align}\label{equ411}
\Delta_{u}^J\big|_{u=0}\cL_{f,d}&(\mu u, a, b,\P_{\alpha})\nonumber\\
&=\sum_{\substack{r,s\ge 0,\ell\ge J\\ 3r+2\ell+2s\le 2d}}
C_{r,\ell,s}(d;f)a^sb^{r}\Delta_{u}^J\big|_{u=0}(u^{\ell})\mu^{\ell}\P_{\alpha}^{r+\ell+2s}=0,
\end{align}
\begin{align}\label{equ412}
\Delta_{u}^J\big|_{u=0}\cL_{f,J}(\mu u,a,b,\P_{\alpha})
=&\sum_{\substack{r,\ell,s\ge 0\\ 3r+2\ell+2s\le 2J}}C_{r,\ell,s}(J;f)a^sb^{r}\Delta_{u}^J\big|_{u=0}(u^{\ell})
\mu^{\ell}\P_{\alpha}^{r+\ell+2s}\nonumber\\
=&J!\mu^JC_{0,J,0}(J;f)\P_{\alpha}^{J};
\end{align}
\begin{align}\label{equ413}
\Delta_{u}^J\big|_{u=0}\cL_{f,J+1}(\mu u,a,b,\P_{\alpha})=&\sum_{\substack{r,\ell,s\ge 0\\ 3r+2\ell+2s\le 2J+2}}C_{r,\ell,s}(J+1;f)a^sb^{r}
\Delta_{u}^J\big|_{u=0}(u^{\ell})\mu^{\ell}\P_{\alpha}^{r+\ell+2s}\nonumber\\
=&\sum_{\substack{\ell,s\ge 0\\ \ell+s\le 1}}C_{0,\ell+J,s}(J+1;f)a^s
\Delta_{u}^J\big|_{u=0}(u^{J+\ell})\mu^{\ell+J}\P_{\alpha}^{J+\ell+2s}\nonumber\\
=&J!\mu^JC_{0,J,0}(J+1;f)\P_{\alpha}^{J}+J!\mu^JC_{0,J,1}(J+1;f)a\P_{\alpha}^{J+2}\nonumber\\
&+\frac{J(J+1)!}{2}\mu^{J+1}C_{0,J+1,0}(J+1;f)\P_{\alpha}^{J+1};
\end{align}
\begin{align}\label{equ414}
\Delta_{u}^J\big|_{u=0}&\cL_{f, J+2}(\mu u, a, b,\P_{\alpha})\nonumber\\
&=\sum_{\substack{r,\ell,s\ge 0\\ 3r+2\ell+2s\le 4}}C_{r,\ell+J,s}(J+2;f)a^sb^r\mu^{\ell+J}\Delta_{u}^J\big|_{u=0}(u^{\ell+J})\P_{\alpha}^{J+r+\ell+2s}\nonumber\\
&=bC_{1,J,0}(J+2;f)J!\mu^J\P_{\alpha}^{J+1}+b^0\left(\dots\right)
\end{align}
and
\begin{align}\label{equ415}
\Delta_{u}^J\big|_{u=0}&\cL_{f, J+3}(\mu u, a, b,\P_{\alpha})\nonumber\\
&=\sum_{\substack{r,\ell,s\ge 0\\ 3r+2\ell+2s\le 6}}C_{r,\ell+J,s}(J+3;f)a^sb^r\mu^{\ell+J}\Delta_{u}^J\big|_{u=0}(u^{\ell+J})\P_{\alpha}^{J+r+\ell+2s}\nonumber\\
&=b^2C_{2,J,0}(J+3;f)J!\mu^J\P_{\alpha}^{J+2}+b^1\left(\dots\right)+b^0\left(\dots\right).
\end{align}
For an integer $d\ge J+4$, we have the following estimate
\begin{align}\label{equ416}
\Delta_{u}^J\big|_{u=0}&\cL_{f, d}(\mu u, a, b,\P_{\alpha})\Big|_{\alpha=\frac{b\beta_f}{2\sqrt{X}}}
\left\{\frac{1}{1+e^{\alpha}}\right\}\nonumber\\
=&\sum_{\substack{r,\ell,s\ge 0\\ 3r+2\ell+2s\le 2d}}C_{r,\ell,s}(d;f)a^sb^{r}\Delta_{u}^J
\big|_{u=0}(u^{\ell})\mu^{\ell}\P_{\alpha}^{r+\ell+2s}\Big|_{\alpha=\frac{b\beta_f}{2\sqrt{X}}}
\left\{\frac{1}{1+e^{\alpha}}\right\}\nonumber\\
\ll &\sum_{\substack{r,s\ge 0,\ell\ge J\\ 3r+2\ell+2s\le 2d}}a^sb^r\left|\P_{\alpha}^{r+\ell+2s}
\Big|_{\alpha=\frac{b\beta_f}{2\sqrt{X}}}\left\{\frac{1}{1+e^{\alpha}}\right\}\right|\nonumber\\
\ll &(1+b^{\lfloor\frac{2d-2J}{3}\rfloor})e^{-\frac{b\beta_f}{2\sqrt{X}}},
\end{align}
by use of Theorem \ref{mth}, Lemma \ref{lem32} and Lemma \ref{lem211}.
Let
\begin{align*}
{M}_{f,J}(a,b,\mu, X,\P_{\alpha})&=J!\mu^JC_{0,J,0}(J;f)\P_{\alpha}^{J}+\frac{J!\mu^J}{\sqrt{X}}\bigg(C_{0,J,0}(J+1;f)\P_{\alpha}^{J}\\
&+aC_{0,J,1}(J+1;f)\P_{\alpha}^{J+2}+\frac{J(J+1)}{2}\mu C_{0,J+1,0}(J+1;f)\P_{\alpha}^{J+1}\bigg)\\
&+J!\mu^J\frac{b}{X}C_{1,J,0}(J+2;f)\P_{\alpha}^{J+1}+J!\mu^J\frac{b^2}{X^{3/2}}C_{2,J,0}(J+3;f)
\P_{\alpha}^{J+2}.
\end{align*}
Inserting \eqref{eqc1}--\eqref{eqc5} yields
\begin{align}\label{eq38}
\frac{{M}_{f,J}(a,b,\mu, X,\P_{\alpha})}{(\mu \beta_f/2)^J}
=&\P_{\alpha}^{J}-\frac{J\left(4\alpha_f+(J-1)\right)}{2\beta_f\sqrt{X}}\P_{\alpha}^{J}\nonumber\\
&-\frac{a\beta_f^2+\left(\frac{\beta_fb}{2\sqrt{X}}\right)^2}{2\beta_f\sqrt{X}}
\P_{\alpha}^{J+2}-\frac{2\left(2\alpha_f+J\right)\left(\frac{\beta_fb}{2\sqrt{X}}\right)
 -\frac{\mu \beta_f^2 J}{2}}{2\beta_f\sqrt{X}}\P_{\alpha}^{J+1}.
\end{align}
Then, by Theorem \ref{mth}, \eqref{equ411}--\eqref{equ415} and the estimate \eqref{equ416} we have
\begin{align}\label{eestimate}
\frac{\Delta_{u}^J\big|_{u=0}S_{f}(a,b+\mu u;X)}{f(X)}&-X^{-J/2}{M}_{f,J}(a,b,\mu, X,\P_{\alpha})
\Big|_{\alpha=\frac{b\beta_f}{2\sqrt{X}}}\left\{\frac{1}{1+e^{\alpha}}\right\}\nonumber\\
\ll &\bigg(\frac{1}{X^{\frac{J+2}{2}}}+\frac{b}{X^{\frac{J+3}{2}}}+\sum_{J+3<d<3J+9}\frac{1+b^{\lfloor\frac{2d-2J}{3}\rfloor}}{X^{d/2}}
   \nonumber\\
    &\qquad\qquad\qquad\qquad+\frac{1+(b/\sqrt{X})^{2J+6}}{X^{\frac{J+3}{2}}}\bigg)e^{-\frac{b\beta_f}{2\sqrt{X}}}\nonumber\\
\ll& X^{-J/2-1}\left(1+(X^{-1/2}b)^4\right)e^{-\frac{b\beta_f}{2\sqrt{X}}}.
\end{align}
Here, the condition $0\le b\le \sqrt{X}(\log X)^2$ has been used in the last inequality of \eqref{eestimate}.

We now aim to simplify \eqref{eq38}. By Taylor's mean value theorem, it is not difficult to prove that, For all $z>-1$ and $\varepsilon=o(1)$, we have
\begin{equation}\label{eqtlem1}
\P_{\alpha}^J\big|_{\alpha=z+\varepsilon}\left\{\frac{1}{1+e^{\alpha}}\right\}
=\P_{\alpha}^J\big|_{\alpha=z}\left\{\frac{1}{1+e^{\alpha}}\right\}+O(|\varepsilon|e^{-z})
\end{equation}
and
\begin{equation}\label{eqtlem2}
\P_{\alpha}^J\big|_{\alpha=z+\varepsilon}\left\{\frac{1}{1+e^{\alpha}}\right\}
=\left(\P_{\alpha}^J+\varepsilon \P_{\alpha}^{J+1}\right)\big|_{\alpha=z}\left\{\frac{1}{1+e^{\alpha}}\right\}+O(|\varepsilon|^2e^{-z}),
\end{equation}
for each $J\ge 0$. Setting $z=\frac{b\beta_f}{2\sqrt{X}}$ and $\varepsilon=\frac{\mu J\beta_f}{4\sqrt{X}}$ in \eqref{eqtlem1} and \eqref{eqtlem2}, and combining \eqref{eq38},
\eqref{eestimate} can be reduced to
\begin{align*}
&\frac{\Delta_u^J\big|_{u=0}S_f(a,b+u\mu;X)}{X^{-J/2}(\mu \beta_f/2)^Jf(X)}\\
&\qquad\qquad=\left({M}_{f,J}(a,b, X,\P_{\alpha})+O\left(\frac{1+X^{-2}b^4}{X}\right)\right)
\Big|_{\alpha=\frac{(2b+\mu J)\beta_f}{4\sqrt{X}}}\left\{\frac{1}{1+e^{\alpha}}\right\},
\end{align*}
for each real number $b\ge 0$ such that $b\le \sqrt{X}(\log X)^2$. Here
\begin{align*}
{M}_{f,J}(a,b,X,\P_{\alpha})=&\P_{\alpha}^{J}-\frac{J\left(4\alpha_f+(J-1)\right)}{2\beta_f\sqrt{X}}\P_{\alpha}^{J}
\\
&-\frac{a\beta_f^2+\left(\frac{\beta_fb}{2\sqrt{X}}\right)^2}{2\beta_f\sqrt{X}}
\P_{\alpha}^{J+2}-\frac{2\left(2\alpha_f+J\right)\left(\frac{\beta_fb}{2\sqrt{X}}\right)}{2\beta_f\sqrt{X}}\P_{\alpha}^{J+1}.
\end{align*}
If we further denote
\begin{equation}\label{eqmmm}
{M}_{f,J}(a, \P_{\alpha})=\left(4\alpha_f-1+J\right)J\P_{\alpha}^{J}
+2(J+2\alpha_f)\alpha\P_{\alpha}^{J+1}+(a\beta_f^2+\alpha^2)\P_{\alpha}^{J+2},
\end{equation}
then
\begin{align*}
&\frac{\Delta_u^J\big|_{u=0}S_f(a,b+u\mu;X)}{X^{-J/2}(\mu \beta_f/2)^Jf(X)}\\
&\qquad\qquad=\left(\P_{\alpha}^J-\frac{{M}_{f,J}(a,\P_{\alpha})}{2\beta_f\sqrt{X}}+O\left(\frac{1+\alpha^4}{X}\right)\right)
\Big|_{\alpha=\frac{(2b+\mu J)\beta_f}{4\sqrt{X}}}\left\{\frac{1}{1+e^{\alpha}}\right\},
\end{align*}
which completes the proof of Theorem \ref{propm} for the cases of $0\le b\le \sqrt{X}(\log X)^2$.

\subsubsection{The cases of $\sqrt{X}(\log X)^2<b=o(X^{3/4})$}\label{sec413}

Assume the conditions for $b$ in Theorem \ref{thm1}, that is $L(X)X^{1/2}\log X\le b\le X-L(X)$ with $L(x)$ a real function satisfying $\lim\limits_{x\rrw+\infty}L(x)=+\infty$.
We have for each $J\in\nb_0$,
\begin{align*}
&\frac{\Delta_u^J\big|_{u=0}S_f(a,b+\mu u;X)}{f(X-b)}\\
&\qquad\qquad\sim \sum_{d\ge 0}\frac{1}{(X-b)^{d/2}}\sum_{\substack{\ell,s\ge 0\\ \ell+s\le d}}
C_{\ell,s}(d;f)(\Delta_u^J\big|_{u=0}u^{\ell})\mu^{\ell}a^{s}\\
&\qquad\qquad=\frac{\mu^JJ!C_{J,0}(J;f)}{(X-b)^{J/2}}+\frac{\mu^JJ!C_{J,0}(J+1;f)}{(X-b)^{(J+1)/2}}+\frac{a\mu^JJ!C_{J,1}(J+1;f)}{(X-b)^{(J+1)/2}}\\
&\qquad\qquad\qquad +\frac{\mu^{1+J} (J+1)!JC_{1+J,0}(J+1;f)}{2(X-b)^{(J+1)/2}}+O\left((X-b)^{-\frac{2+J}{2}}\right),
\end{align*}
by Theorem \ref{thm1}. Inserting the values of $C_{r,\ell}(d;f)$, that is \eqref{eqc11}--\eqref{eqc31}, we find that
\begin{align*}
&\frac{\Delta_u^J\big|_{u=0}S_f(a,b+\mu u;X)}{f(X-b)}\\
&\quad=\frac{(-\mu)^JJ!}{(X-b)^{\frac{J}{2}}}\frac{\beta_f^J}{2^JJ!}
 -\frac{(-\mu)^JJ!}{(X-b)^{\frac{J+1}{2}}}\left(\frac{\beta_f^{J-1}\alpha_f}{2^{J-1}\Gamma(J)}
 +\frac{\beta_f^{J-1}}{2^{J+1}\Gamma(J-1)}\right)\\
&\quad\quad-\frac{a(-\mu)^JJ!}{(X-b)^{\frac{J+1}{2}}}\frac{\beta_f^{J+1}}{2^{J+1}J!}
+\frac{(-\mu)^{1+J} (J+1)!J}{2(X-b)^{\frac{J+1}{2}}}\frac{\beta_f^{J+1}}{2^{J+1}(J+1)!}+O\left(\frac{1}{(X-b)^{\frac{2+J}{2}}}\right).
\end{align*}
Further simplification yields
\begin{align}\label{eq39}
&\frac{\Delta_u^J\big|_{u=0}S_f(a,b+\mu u;X)}{(X-b)^{-J/2}(-\mu\beta_f/2)^Jf(X-b)}\nonumber\\
&\qquad\qquad\qquad=1-\frac{J\left(4\alpha_f-1+J\right)
 +\frac{\mu J\beta_f^2}{2}+a\beta_f^2}{2\beta_f\sqrt{X-b}}+O\left(\frac{1}{X-b}\right).
\end{align}
From now on we assume $L(X)X^{1/2}\log X\le b=o(X^{3/4})$. Then we have
\begin{align*}
\frac{\Delta_u^J\big|_{u=0}S_f(a,b+\mu u;X)}{X^{-J/2}(-\mu\beta_f/2)^Jf(X)}=&\left(1-\frac{J\left(4\alpha_f-1+J\right)
+\frac{\mu J\beta_f^2}{2}+a\beta_f^2}{2\beta_f\sqrt{X}}+O\left(\frac{b}{X^{3/2}}\right)\right)\\
&\times
\left(1-\frac{b}{X}\right)^{-{J}/{2}}\frac{f(X-b)}{f(X)}
\end{align*}
by \eqref{eq39}. Since
$$\left(1-\frac{b}{X}\right)^{-{J}/{2}}=1+\frac{2J}{2\beta_f\sqrt{X}}\left(\frac{\beta_fb}{2\sqrt{X}}\right)+O\left(\frac{b^2}{X^2}\right),$$
and by Proposition \ref{lem22},
\begin{align*}
\frac{e^{\frac{\beta_fb}{2\sqrt{X}}}f(X-b)}{f(X)}
&=1-\frac{b}{X^{3/4}}\Lambda_{1}(f, X)+\frac{b^2}{X^{3/2}}\Lambda_{2}(f, X)
+O\left(\frac{b^3}{X^{5/2}}+\frac{b^4}{X^3}\right)\\
&=1-\frac{b}{X}\lambda_{1,1}(f)+\frac{b^2}{X^{3/2}}\lambda_{0,2}(f)
+O\left(\frac{b}{X^{3/2}}+\frac{b^2}{X^{2}}+\frac{b^4}{X^3}\right)\\
&=1+\frac{4\alpha_f\left(\frac{\beta_fb}{2\sqrt{X}}\right)-\left(\frac{\beta_f b}{2\sqrt{X}}\right)^2}{2\beta_f\sqrt{X}}
+O\left(\frac{b^4}{X^3}\right),
\end{align*}
we further have
\begin{align}\label{eq40}
&\frac{\Delta_u^J\big|_{u=0}S_f(a,b+\mu u;X)}{X^{-J/2}(-\mu\beta_f/2)^Jf(X)}\nonumber\\
&\qquad\qquad=\left(1-\frac{\mu J\beta_f}{4\sqrt{X}}\right)e^{-\frac{\beta_fb}{2\sqrt{X}}}
+O\left(\frac{X^{-2}b^4}{X}e^{-\frac{\beta_fb}{2\sqrt{X}}}\right)\nonumber\\
&\qquad\qquad-\frac{J\left(4\alpha_f-1+J\right)+a\beta_f^2+\left(\frac{\beta_f b}{2\sqrt{X}}\right)^2
-2\left(2\alpha_f+J\right)\left(\frac{\beta_fb}{2\sqrt{X}}\right)}{2\beta_f\sqrt{X}}e^{-\frac{\beta_fb}{2\sqrt{X}}}.
\end{align}
Finally, by inserting
$$e^{-\frac{\beta_fb}{2\sqrt{X}}}=\left(1+\frac{\mu J\beta_f}{2\sqrt{X}}
+O\left(\frac{1}{X}\right)\right)e^{-\frac{\beta_f(2b+\mu J)}{2\sqrt{X}}}$$
and using the fact that for each $N\in\nb_0$ and $b\ge \sqrt{X}$,
$$e^{-\frac{\beta_f(2b+\mu J)}{2\sqrt{X}}}=\left(O\left(e^{-\alpha}\right)+(-1)^N\P_{\alpha}^N\right)
\bigg|_{\alpha=\frac{\beta_f(2b+\mu J)}{2\sqrt{X}}}\left\{\frac{1}{1+e^{\alpha}}\right\},$$
in \eqref{eq40}, and simplifying, it is not difficult to find that
\begin{align*}
&\frac{\Delta_u^J\big|_{u=0}S_f(a,b+u\mu;X)}{X^{-J/2}(\mu \beta_f/2)^Jf(X)}\\
&\qquad\qquad=\left(\P_{\alpha}^J-\frac{{M}_{f,J}(a,\P_{\alpha})}{2\beta_f\sqrt{X}}+O\left(\frac{1+\alpha^4}{X}\right)\right)
\Big|_{\alpha=\frac{(2b+\mu J)\beta_f}{4\sqrt{X}}}\left\{\frac{1}{1+e^{\alpha}}\right\},
\end{align*}
where ${M}_{f,J}(a,\P_{\alpha})$ defined by \eqref{eqmmm}. This completes the proof of Theorem \ref{propm} for the cases of
$\sqrt{X}(\log X)^{2}< b=o(X^{3/4})$.

\subsection{The proof of Theorem \ref{propm'}}\label{sec42}

In this subsection we prove Theorem \ref{propm'}, we assume that $J\in\{0,1,2\}$. In view of Theorem \ref{propm} we need:
$$
\P_{\alpha}^{0}\left\{\frac{1}{1+e^{\alpha}}\right\}=\frac{1}{2}\left(1-\tanh\left(\frac{\alpha}{2}\right)\right),
$$
$$
\P_{\alpha}^{1}\left\{\frac{1}{1+e^{\alpha}}\right\}=-\frac{1}{4}{\rm sech}^2\left(\frac{\alpha}{2}\right),
$$
$$
\P_{\alpha}^{2}\left\{\frac{1}{1+e^{\alpha}}\right\}=\frac{1}{4}{\rm sech}^2\left(\frac{\alpha}{2}\right)\tanh\left(\frac{\alpha}{2}\right)$$
and
$$
\P_{\alpha}^{3}\left\{\frac{1}{1+e^{\alpha}}\right\}=\frac{1}{8}{\rm sech}^4\left(\frac{\alpha}{2}\right)\left(1-2\sinh^2\left(\frac{\alpha}{2}\right)\right).
$$

From the above computations, for any given positive number $c>0$ and $J\in\{0,1,2\}$ we have that
\begin{equation}\label{eq250}
\min(1, |\alpha|)\ll e^{\alpha}\P_{\alpha}^J\left\{\frac{1}{1+e^{\alpha}}\right\}\ll \min(1, |\alpha|),
\end{equation}
holds uniformly for all $\alpha>-c$.
On the other hand, for any $b\ge 0$ with $b=o(X^{3/4})$, from Theorems \ref{propm} and \ref{lem211} we have
\begin{equation}\label{eq251}
\frac{\Delta_u^J\big|_{u=0}S_f(a,b+u\mu;X)}{X^{-J/2}(\mu \beta_f/2)^Jf(X)}
=\bigg(1+O\bigg(\frac{1+\alpha^2}{\sqrt{X}}\bigg)\bigg)\P_{\alpha}^J
\Big|_{\alpha=\frac{(2b+\mu J)\beta_f}{4\sqrt{X}}}\bigg\{\frac{1}{1+e^{\alpha}}\bigg\}.
\end{equation}
Letting $X\mapsto X+b$, noting that (by Proposition~\ref{lem22})
$$f(X+b)\sim e^{\frac{\beta_f b}{2\sqrt{X}}}f(X),$$
and using \eqref{eqtlem1} to \eqref{eq251}, as well as the condition that $|\mu J+2b|\ge c>0$ in \eqref{eq250}, we obtain
$$
\frac{\Delta_u^J\big|_{u=0}S_f(a,b+u\mu;X+b)}{X^{-J/2}(\mu \beta_f/2)^Jf(X)}
\sim e^{\alpha}\P_{\alpha}^J
\Big|_{\alpha=\frac{(2b+\mu J)\beta_f}{4\sqrt{X}}}\left\{\frac{1}{1+e^{\alpha}}\right\}.
$$
Moreover, using \eqref{eq39} we have
\begin{align*}
\frac{\Delta_u^J\big|_{u=0}S_f(a,b+\mu u;X+b)}{X^{-J/2}(-\mu\beta_f/2)^Jf(X)}&=1+O\left(\frac{1}{\sqrt{X}}\right)\nonumber\\
&=e^{\alpha}\left(\P_{\alpha}^J+O\left(\frac{1}{\sqrt{X}}\right)\right)\bigg|_{\alpha=\frac{(2b+\mu J)\beta_f}{2\sqrt{X}}}\left\{\frac{1}{1+e^{\alpha}}\right\}.
\end{align*}
This implies that for
$$L(X)X^{1/2}\log X\le b\le X-L(X)\; \text{with} \; \lim_{x\rrw+\infty}L(x)=+\infty,$$
we have
$$\frac{\Delta_u^J\big|_{u=0}S_f(a,b+u\mu;X+b)}{X^{-J/2}(\mu \beta_f/2)^Jf(X)}
\sim e^{\alpha}\P_{\alpha}^J
\Big|_{\alpha=\frac{(2b+\mu J)\beta_f}{4\sqrt{X}}}\left\{\frac{1}{1+e^{\alpha}}\right\}.
$$
Combining the above we complete the proof of Theorem \ref{propm'}.

\section{The proof of the remaining results and numerical data}

In this section we prove use Theorem \ref{propm} and Theorem \ref{propm'} to prove the remaining results of this paper.
To verify our asymptotic formula we also illustrate some numerical data.

\subsection{The proof of Theorems \ref{znh1}--\ref{znh5}}\label{sec5}

\begin{table}
	\renewcommand{\arraystretch}{1.2}
        \caption{Numerical data for $b_{m,1}(n)$.}
	\label{tab1}
	\centering
  \begin{tabular}{ c | c | c | c  }
\hline

    $n$ & $b_{1,1}(n^2)$ & $T(1,n^2)$ & ${b_{1,1}(n^2)}/{T(1,n^2)}$\\
\hline
   $50$ & $8.67687\cdot 10^{45}$ & $9.08059\cdot 10^{45}$ & $0.9555$ \\
   $100$ & $1.39866\cdot 10^{100}$ & $1.43049\cdot 10^{100}$ & $0.9777$ \\
   $200$ & $1.11517 \cdot 10^{210}$ & $1.12772\cdot 10^{210}$ & $0.9889$ \\
   $400$ & $2.22252\cdot 10^{431}$ & $2.23496 \cdot 10^{431}$ & $0.9944$ \\
\hline
\hline
    $n$ & $b_{n,1}(n^2)$  & $T(n,n^2)$ & ${b_{n,1}(n^2)}/{T(n,n^2)}$\\
\hline
   $50$ & $1.77991\cdot 10^{47}$ & $1.81723\cdot 10^{47}$ & $0.9795$ \\
   $100$ & $5.66389\cdot 10^{101}$ & $5.72242\cdot 10^{101}$ & $0.9898$ \\
   $200$ & $8.97474 \cdot 10^{211}$ & $9.02079\cdot 10^{211}$ & $0.9949$ \\
   $400$ & $3.56615\cdot 10^{433}$ & $3.57527 \cdot 10^{433}$ & $0.9974$ \\
\hline
  \end{tabular}
\end{table}

 We first prove Theorem \ref{znh1} and \ref{znh2}. From \eqref{eq00}, that is
$$j_{m,k}(n-m{\bf 1}_{m>0})=S_{p_k}(1/2,|m|-1/2;n),$$
the fact that $a_{m,k}(n)=a_{|m|,k}(n)$, and \eqref{eqck} we obtain
\begin{align*}
a_{m,k}(n)&=a_{-m,k}(n)\\
&=j_{-m,k}(n)-j_{-m-1,k}(n)\\
&=\Delta_u\big|_{u=0}S_{p_k}(1/2,m+1/2-u;n).
\end{align*}
Therefore, by using \eqref{eqck} we find that
$$
b_{m,k}(n)=\Delta_u^2\big|_{u=0}S_{p_k}(1/2,m+3/2-u;n),
$$
holds for all integers $m, n\ge 0$.  Thus by substituting $\beta_{p_k}=2\pi\sqrt{k/6}$ (see \eqref{aspk}) into Theorems \ref{propm} and \ref{propm'},
and above relations for $j_{m,k}(n-m{\bf 1}_{m>0}), a_{m,k}(n)$ and $b_{m,k}(n)$ we immediately get the proof of Theorems \ref{znh1} and \ref{znh2}.

\smallskip

We now prove Theorems \ref{znh3}--\ref{znh5}.  From \eqref{eq00'} and \eqref{2grk} we have
$$I_k(m,n)=S_{p}(k-1/2,m-1/2;n),$$
holds for all integers $m,n\ge 0$, and for all $m\in\zb$ and integers $n\ge 0$,
\begin{align*}
N_k(m,n)&=I_k(|m|,n)-I_k(|m|+1,n)\\
&=\Delta_u\big|_{u=0}S_{p}(k-1/2,|m|+1/2-u;n).
\end{align*}
This means that for all integers $n\ge 0$,
$$
N_k(m,n)-N_k(m+1,n)=
\begin{cases}
\Delta_u^2\big|_{u=0}S_{p}(k-1/2,m+3/2-u;n),\quad &~ m\ge 0 \\
-[N_k(|m|-1,n)-N_k(|m|,n)]&~ m<0.
\end{cases}
$$
\begin{table}
	\renewcommand{\arraystretch}{1.2}
        \caption{Numerical data for $N_2(m,n)-N_2(m+1,n)$.}
	\label{tab1}
	\centering
  \begin{tabular}{ c | c | c | c  }
\hline
    $n$ & $N_2(0,n^2)-N_2(1,n^2)$ & $T(0,n^2)$ & $\frac{N_2(0,n^2)-N_2(1,n^2)}{T(0,n^2)}$\\
\hline
  $50$ & $3.04871\cdot 10^{45}$ & $3.02819\cdot 10^{45}$ & $1.0068$ \\
   $100$ & $4.78500\cdot 10^{99}$ & $4.76884\cdot 10^{99}$ & $1.0034$ \\
   $200$ & $3.76555 \cdot 10^{209}$ & $3.75918\cdot 10^{209}$ & $1.0017$ \\
   $400$ & $7.45623\cdot 10^{430}$ & $7.44992\cdot 10^{430}$ & $1.0008$ \\
\hline
\hline
    $n$ & $N_2(n+1,n^2)-N_2(n+2,n^2)$  & $T(n+1,n^2)$ & $\frac{N_2(n+1,n^2)-N_2(n+2,n^2)}{T(n+1,n^2)}$\\
\hline
   $50$ & $1.82908\cdot 10^{47}$ & $1.81764\cdot 10^{47}$ & $1.0063$ \\
   $100$ & $5.74203\cdot 10^{101}$ & $5.72403\cdot 10^{101}$ & $1.0031$ \\
   $200$ & $9.03662 \cdot 10^{211}$ & $9.02244\cdot 10^{211}$ & $1.0016$ \\
   $400$ & $3.57845\cdot 10^{433}$ & $3.57564\cdot 10^{433}$ & $1.0008$ \\
\hline
  \end{tabular}
\end{table}
Thus, by substituting $\beta_{p}=2\pi/\sqrt{6}$ (see \eqref{aspk}) into  Theorems \ref{propm} and \ref{propm'},
and above relations for $I_k(m,n)$ and $N_k(m,n)$, it is not difficult to obtain the proof of Theorems \ref{znh3}--\ref{znh5}.

\medskip

\subsection{Numerical data}
As stated in this article, the leading uniform asymptotics of $a_{m,k}(n)$ and $N_2(m,n)$ have been proved in existing literature.
While the leading uniform asymptotics for $b_{m,k}(n)$ and $N_{k}(m,n)-N_k(m+1,n)$ are completely new in this paper.

\smallskip

From Theorem \ref{znh1} and \ref{znh3}, we find that the leading asymptotics of both $b_{m,k}(n)$ and $N_{k}(m,n)-N_k(m+1,n)$ are the same, and equal
$$
T(m,n)=\frac{\pi^2}{24n}{\rm sech}^2\left(\frac{\pi(2m+1)}{4\sqrt{6n}}\right)\tanh \left(\frac{\pi(2m+1)}{4\sqrt{6n}}\right) p(n).
$$

We illustrate some of our results in the following tables (All computations are done in {\bf Mathematica}, and they are all approximate values.)
\paragraph{Acknowledgements.} The author would like to thank the anonymous referees for their very helpful comments and suggestions. This work was supported by the National Science Foundation of China (Grant No. 11971173) and Science and Technology Commission of Shanghai Municipality (Grant No. 13dz2260400).

\bigskip
\noindent
{\sc Zhi-Guo Liu and Nian Hong Zhou\\
School of Mathematical Sciences \& Shanghai Key Laboratory of PMMP\\
East China Normal University\\
500 Dongchuan Road, Shanghai 200241, People's Republic of China}\newline
E-mails: \href{mailto:nianhongzhou@outlook.com}{\small nianhongzhou@outlook.com};\; \href{mailto:zgliu@math.ecnu.edu.cn, liuzg@hotmail.com}{\small zgliu@math.ecnu.edu.cn, liuzg@hotmail.com}

\end{document}